\documentclass[11pt]{article}
\usepackage[english]{babel}
\usepackage[latin1]{inputenc}
\usepackage[T1]{fontenc}
\usepackage{amsmath,amssymb,amsthm,xspace}

\RequirePackage{ifpdf}
\ifpdf
   \usepackage[pdftex]{hyperref}
\else
   \usepackage[hypertex]{hyperref}
   \usepackage{srcltx}
\fi

\usepackage{graphicx}

\title{Ampleness in the free group
}
\author{A. Ould Houcine\thanks{ The author was supported by SFB 878. }\ \  and K. Tent}
\date{\today}

\renewcommand{\>}{\rangle}
\newcommand{\<}{\langle}
\newcommand{\tfh}{torsion-free hyperbolic }
\newtheorem{theorem}{Theorem}[section]

\newtheorem{lemma}[theorem]{Lemma}
\newtheorem{proposition}[theorem]{Proposition}
\newtheorem{corollary}[theorem]{Corollary}
\newtheorem{definition}[theorem]{Definition}
\newtheorem{remark}[theorem]{Remark}

\newcommand{\nc}{\newcommand}
\nc{\inv}{^{-1}}
\nc{\Q}{\mathbb{Q}}
\nc{\R}{\mathbb{R}}
\nc{\C}{\mathcal{C}}
\nc{\G}{\Gamma}
\nc{\M}{\mathcal{M}}
\nc{\U}{\mathbb{U}}
\nc{\E}{\mathcal E}

\renewcommand{\phi}{\varphi}

\renewcommand{\b}{\beta}

\renewcommand{\a}{\alpha}

\DeclareMathOperator{\acl}{acl}

\DeclareMathOperator{\acleq}{acl^{eq}}

\newcommand{\Ind}{
 \setbox0=\hbox{$x$}\kern\wd0\hbox to 0pt{\hss$
 \mid$\hss}\lower.9\ht0\hbox to 0pt{\hss$\smile$\hss}\kern\wd0
}
\newcommand{\indep}[3]{#1\mathop{\mathpalette\Ind{}}_{#2}#3
}

\def\NotInd#1#2{#1\setbox0=\hbox{$#1x$}\kern\wd0\hbox to 0pt{\mathchardef
\nn="3236\hss$#1\nn$\kern1.4\wd0\hss}\hbox to
0pt{\hss$#1\mid$\hss}\lower.9\ht0 \hbox to
0pt{\hss$#1\smile$\hss}\kern\wd0}

\begin{document}
\maketitle

\begin{abstract} We show that the theory of the free group 
-- and more generally the theory of any torsion-free hyperbolic group -- is $n$-ample for any $n\geq 1$. We give also  an explicit description of the imaginary algebraic closure in free groups. 
\end{abstract}

\section{Introduction}

Work of Kharlampovich-Myasnikov  \cite{khar-mia} and Sela \cite{sela-el}  showed that the theory of a nonabelian free group does not depend on the rank of the free group and so we can let $T_{fg}$
denote the theory of nonabelian free groups.
Sela~\cite{sel-ima} showed that this theory   is stable as well   as  the theory of any torsion-free hyperbolic group.

Having a quantifier elimination result down to $\forall\exists$-formulas \cite{khar-mia, sela-el}, 
elimination of imaginaries down to some very restricted class of imaginaries \cite{sel-ima}, homogeneity \cite{hom-ould, perin-homo}, a better understanding of stability theoretic independence \cite{pillay-free,pillay-weight} and more recently a description of the algebraic closure \cite{OHV} now gives us the tools for studying the model theoretic geometry of forking in
the free group and in torsion-free hyperbolic groups. 

Ampleness is a property that reflects the existence of geometric configurations behaving very much like projective space over a field. Pillay \cite{pillay-amp} first defined the notion of $n$-ampleness. We use here the slightly stronger definition given by Evans in \cite{david-ample}. 

\begin{definition} \emph{\cite{david-ample}} \label{def-ample} Suppose $T$ is a complete stable theory  and $n \geq 1$ 
 is a natural number. Then $T$ is $n$-ample if (in some model of $T$, possibly after naming some 
parameters) there exist tuples $a_0 , \dots, a_n$  such that: 

\smallskip

\noindent $(i)$  ${a_n}  \NotInd{}{} {a_0}$;

\noindent $(ii)$  $a_0 \ldots a_{i-1} \Ind_{a_i} a_{i+1}\ldots a_n$ for $1 \leq i <n$; 

\noindent $(iii)$ $\acleq(a_0 ) \cap \acleq(a_1) = \acleq(\emptyset)$; 

\noindent $(iv)$   $\acleq(a_0 \dots a_{i-1} a_i ) \cap \acleq(a_0 \dots  a_{i-1} a_{i+1}) = \acleq(a_0 \dots a_{i-1})$ for  $1\leq  i < n$. 

\medskip
\noindent
We call $T$ \emph{ample} if it is $n$-ample for all $n\geq 1$.
\end{definition}

In $n+1$-dimensional projective space such a tuple $a_0,\ldots a_n$ can be
chosen as a maximal \emph{flag} of subspaces and this example can guide the intuition.  It is well-known that 
a stable structure which type-interprets an infinite field is ample.  In this  paper  we study ampleness in torsion-free hyperbolic groups and we show:

\begin{theorem}\label{T-main} The theory of  any nonabelian torsion-free hyperbolic group is ample. 
\end{theorem}

We also state the following special case, whose proof --- as we will see --- implies the
general statement:
\begin{corollary}
 The theory $T_{fg}$ of nonabelian free  groups is ample. 
\end{corollary}

In fact, our construction is very explicit: given a basis
\[\{c_0,a_i,b_i,t_i\colon i<\omega\}\] of the free group of rank $\omega$
we find witnesses for ampleness as the tuples
$h_{2i}=(a_{2i},b_{2i},c_{2i}), i<\omega$, where the sequence $(c_i)_{i<\omega} $ is defined inductively as follows \[c_{i+1}=t_i c_i^{-1}[a_i,b_i]^{-1}{t_i}^{-1}.\]

We do not know whether the free group interprets an infinite field. While
we do not believe this to be the case, ampleness is certainly consistent with
the existence of a field.

In \cite{pillay-free} Pillay gave a proof showing that the free group is $2$-ample. However, his
proof relied on a result by Bestvina and Feighn which has not yet been
completely established. His  conjecture that the free group is \emph{not}  $3$-ample is refuted in this paper.

The present paper is organized as follows.  In Section 2 we collect the preliminaries about elimination of imaginaries,  JSJ-decompositions and related notions  needed in the sequel. In Section 3, we study the imaginary algebraic closure and give a geometric characterization of those conjugacy classes which are elements of the imaginary algebraic closure. 
Section 4 is   devoted to the construction of sequences witnessing the ampleness in the free group. The last section then shows how the general theorem
follows from the special case. 

\section{Preliminaries}

In this section we put together some  background material needed in the sequel. The first subsection deals with imaginaries and the two next subsections deal with  splittings,  JSJ-decompositions, homogeneity,  algebraic closure and independence.

\subsection{Imaginaries}

Let $T$ be a complete theory and $M$ a (very saturated) model of $T$.
Recall that $T$ has \emph{geometric elimination of imaginaries}
if for any $\emptyset$-definable equivalence relation $E$ on $M^k$ and any 
equivalence class $a_E\in M^{eq}$ there is a finite tuple $\bar a\subset M$
with $\acleq(\bar a)=\acleq(a_E)$, i.e., $a_E\in \acleq(\bar a)$ and $\bar a\in \acleq(a_E)$.

 For a subset $\mathcal E$ 
of the set of $\emptyset$-definable equivalence relations in $T$, we let $M_{\mathcal E}$ denote the restriction of $M^{eq}$ to the sorts in $\mathcal E$. That is, for every $E \in \mathcal E$ defined on $M^k$, we add a new sort $S_E$ to the language interpreted as $M^k/E$ and a new function $\pi_E : M^k \rightarrow S_E$ which associates to a $k$-tuple $x\in M^k$ its equivalence class $x_E$.  Then
the $L_{\mathcal E}$-structure $M_{\mathcal E}$ is the disjoint union of $M/E$ for $E\in\mathcal E$. 
We say that $T$ has geometric elimination of imaginaries \emph{relative
to $\E$} if $M_\E$ has geometric elimination of imaginaries.

\begin{remark} \label{lem-link} Suppose that $T$ has geometric elimination of imaginaries relative to $\mathcal E$. Then for tuples $\bar a, \bar b,\bar c\subset M$ we have

\[
\acleq(\bar a)\cap\acleq(\bar b)= \acleq(\bar c)\]
 if and only if 
\[ (\acleq(\bar a) \cap M_{\mathcal E})\cap (\acleq(\bar b) \cap M_{\mathcal E})= (\acleq(\bar c) \cap M_{\mathcal E}).
\]
\end{remark}

In order to prove ampleness for torsion-free hyperbolic groups we can therefore restrict our attention to some basic equivalence relations. For the theory of a nonabelian torsion-free hyperbolic group
Sela established geometric elimination of imaginaries relative to the following small collection of basic equivalence relations:
let $\mathcal E$ denote the collection of the following $\emptyset$-definable equivalence relations where $C(x)$ denotes the centralizer of $x$ and $m,p,q$ are positive natural numbers: 
$$E_0(x;y) : \exists z x^z=y$$
$$E_{1,m}(x,y; x',y'):  C(x)=C(x') \wedge \exists t \in C(x) \hbox{ such that } y' =yt^m$$
$$E_{2,m}(x,y; x',y'): C(x)=C(x') \wedge \exists t \in C(x) \hbox{ such that }y'=t^my $$
$$E_{3, p, q}(x,y,z; x',y',z'): C(x)=C(x') \wedge C(y')=C(y) \wedge \exists s \in C(x) \wedge \exists t \in C(y) \hbox{ such that }z=s^pz't^q.$$

For an $L$-structure $M$ we denote by $\mathcal P_n(M^m)$ the set of finite subsets of $M^m$ of cardinality at most $n$.  A function $f : M^r \rightarrow \mathcal P_n(M^m)$ is said to be definable if there exists a formula $\varphi(\bar x; \bar y)$ such that for any $\bar a \in M^r$, for any $\bar b \in M^m$, $\bar b \in f(\bar a)$ if and only if $M \models \varphi(\bar a, \bar b)$.

Sela proves the following strong form of geometric elimination of imaginaries:
\begin{theorem}\emph{\cite[Theorem 4.6]{sel-ima}}\label{sela-ima} For any $\emptyset$-definable equivalence relation $E(\bar x, \bar y)$, $| \bar x| =n$ in the theory $T$ of a torsion-free hyperbolic group $M$ (in the language of groups),  there exist $k , p\in \mathbb N$ and a function $f : M^n \rightarrow P_p({M_{\mathcal E}}^k)$  $\emptyset$-definable in $L_\E$ such that for all $\bar a, \bar b \in M^n$

$$
M \models E(\bar a, \bar b) \Leftrightarrow M_{\mathcal E} \models f(\bar a)=f(\bar b). 
$$ 
\end{theorem}

Clearly, this implies that for any $0$-definable equivalence relation $E$ in $T$ with corresponding definable function $f_E$ in $L_\E$, the equivalence class $\bar a_E$ is interalgebraic in $M_\E$ with the finite set $f_E(\bar a)$ of elements in $M_\E$: namely, if $f_E(\bar a)=(\bar c_1,\ldots \bar c_k)$ where the $\bar c_i$ are tuples from $M_\E$,
then  $\bar c_1,\ldots,\bar c_k\in \acleq(\bar a_E)$ and
$\bar a_E\in\acleq(\bar c_1,\ldots,\bar c_k)$.
Hence:

\begin{corollary} \emph{\cite{sel-ima}} \label{cor-sela-ima} The theory of a torsion-free hyperbolic group has geometric elimination of imaginaries relative to $\mathcal E$. \qed
\end{corollary}

\subsection{Splittings, JSJ-decompositions}

Let $G$ be a group and let $\mathcal C$ be a class of subgroups of $G$. By a \textit{$(\mathcal C, H)$-splitting} of $G$ (or a splitting of $G$ over $\mathcal C$ relative to $H$), we understand a tuple $\Lambda=(\mathcal G(V,E), T, \varphi)$, where $\mathcal G(V,E)$ is a graph of groups such that each edge group is in $\mathcal C$ and  $H$ is elliptic, $T$ is a maximal subtree of $\mathcal G(V,E)$ and $\varphi : G \rightarrow \pi(\mathcal G(V,E), T)$ is an isomorphism; here $\pi(\mathcal G(V,E), T)$ denotes the fundamental group of $\mathcal G(V, E)$ relative to $T$. If $\mathcal C$ is the class of abelian groups or cyclic groups, we will just say \textit{abelian splitting} or \textit{cyclic splitting}, respectively.   If every edge group is malnormal in the adjacent vertex groups, then we say that the splitting is  \textit{malnormal}. 

Given a group  $G$  and a subgroup $H$  of $G$,  $G$ is said to be \textit{freely $H$-decomposable} if $G$ has a nontrivial free decomposition $G=G_1*G_2$ such that $H \leq G_1$.  Otherwise, $G$ is said to be freely \textit{$H$-indecomposable}.

Following \cite{jsj-gl}, given a group $G$ and two $(\C, H)$-splittings $\Lambda_1$ and $\Lambda_2$ of  $G$,  $\Lambda_1$ \textit{dominates} $\Lambda_2$ if every subgroup of $G$ which is elliptic in $\Lambda_1$ is also elliptic in $\Lambda_2$.  A $(\C, H)$-splitting of $G$ is said to be \textit{universally elliptic} if all edge stabilizers in $\Lambda$ are elliptic in any other $(\C, H)$-splitting of $G$.

A \textit{JSJ-decomposition of $G$ over $\C$ relative to  $H$} is an universally elliptic  $(\C, H)$-splitting   dominating all other universally elliptic $(\C,H)$-splittings.  If $\C$ is the class of abelian subgroups, then we  simply say \textit{abelian JSJ-decomposition}; similarly when $\C$ is the class of cyclic subgroups. It follows from \cite{sela-JSJ, jsj-gl} that torsion-free hyperbolic groups (so in particular nonabelian free groups of finite rank) admit (relative) cyclic JSJ-decompositions. 

Given an abelian  splitting $\Lambda$  of $G$ (relative to $H$) and a vertex group $G_v$ of $\Lambda$, the \textit{elliptic abelian neighborhood}  of $G_v$ is the subgroup generated by the elliptic elements that commute with nontrivial elements of $G_v$.  It was shown in \cite[Proposition 4.26]{gui-limit} that if $G$ is commutative transitive then any  abelian splitting $\Lambda$  of $G$ (relative to $H$) can be transformed to an abelian splitting $\Lambda'$ of $G$ such that the underlying graph is the same as that of $\Lambda$ and for any vertex $v$, the corresponding new vertex group $\hat G_v$ in $\Lambda'$ is the elliptic abelian neighborhood of $G_v$ (similarly for edges); in particular any edge group of $\Lambda'$ is malnormal in the adjacent vertex groups. We call that transformation the \textit{malnormalization} of $\Lambda$.  If $\Lambda$ is a (cyclic or abelian) JSJ-decomposition of $G$ and $G$ is commutative transitive then the malnormalization of $\Lambda$ will be called a \textit{malnormal} JSJ-decomposition. If $G_v$ is a rigid vertex group then we call $\hat G_v$ also rigid; similarly for abelian and surface type vertex groups. Strictly speaking a malnormal JSJ-decomposition  is not a JSJ-decomposition in the sense  of \cite{jsj-gl}, however it possesses the most important properties of JSJ-decompositions that we need. 

We end with the definition of   \textit{generalized malnormal  JSJ-decomposition} relative to a subgroup $A$. First, split  $G$ as a free product $G=G_1*G_2$, where $A \leq G_1$ and $G_1$ is freely $A$-indecomposable. Then, define a  generalized malnormal (cyclic) JSJ-decomposition of $G$ relative to $A$ as the (cyclic) splitting obtained by adding $G_2$ as a new vertex group to a malnormal (cyclic) JSJ-decomposition of $G_1$ (relative to $A$). We call $G_2$ the \textit{free factor}.  In a similar way the notion of a \textit{generalized cyclic JSJ-decomposition}, without the assumption of malnormality,  is defined

We denote by $Aut_H(G)$ the group of automorphisms of a group $G$ that fix 
a subgroup $H$ pointwise. The \textit{abelian modular group} of $G$ relative to $H$, denoted $Mod_H(G)$, is the subgroup of $Aut_H(G)$ generated by Dehn twists, modular automorphisms of abelian type and modular automorphisms of surface type (For more details we refer the reader for instance to \cite{OHV}).  We will use the following property of modular automorphisms whose proof is essentially contained in \cite{gui-limit} (see for instance \cite[Proposition 4.18]{gui-limit}). 

\begin{lemma}\label{lem-modular} Let $\Gamma$ be torsion-free hyperbolic group and $A$ a nonabelian  subgroup of $\Gamma$.  Let $\Lambda$ be a generalized malnormal cyclic JSJ-decomposition of $\Gamma$ relative to $A$. Then any modular automorphisms $\sigma \in Mod_A(\Gamma)$ retrict to a conjugation on rigid vertex groups  and on boundary subgroups of surface type vertex groups of $\Lambda$. \qed
\end{lemma}

We note that when $\Gamma$ is a torsion-free hyperbolic group which is freely $A$-indecomposable, then any abelian vertex group in any cyclic JSJ-decomposition $\Lambda$ of $\Gamma$ relative to $A$ is rigid; this a consequence of the fact that abelian subgroups of $\Gamma$ are cyclic and of the fact that edge groups of $\Lambda$ are universally elliptic.  The same property holds also for abelian vertex groups (different from  free factors) in  generalized malnormal JSJ-decompositions. 

Rips and Sela showed that the 
modular group has finite index in the group of automorphisms. We will use that result in the relative case. 

\begin{theorem} \emph{(See for instance \cite{OHV})} \label{mod-index} Let $\Gamma$ be a torsion-free hyperbolic group and $A$ a nonabelian  subgroup of $\Gamma$ such that $\Gamma$ is freely $A$-indecomposable.  Then $Mod_A(\Gamma)$ has finite index in $Aut_A(\Gamma)$. \qed
\end{theorem}

We end this subsection with the following standard lemma needed in the sequel.

\begin{lemma} \label{split-surface}Let $G$ be the fundamental group of an orientable  surface $S$ with boundaries such that $2g(S)+b(S) \geq 4$, where here $g(S)$ denotes the genus of $S$ and $b(S)$ denotes the number of boundary components. Then for any nontrivial element $g$ which is not conjugate to any element of a boundary subgroup, there exists a malnormal cyclic splitting of $G$ in which $g$ is hyperbolic and boundary subgroups  are elliptic. 
\end{lemma}

\begin{proof} The proof is by induction on $g(S)$. If $g(S)=0$, then $G=\<s_1, \cdots, s_n| s_1 \cdots s_n=1\>$ and $n \geq 4$.  Note that $s_1, \cdots, s_{n-1}$ is a basis of $G$. Since $g$ is not conjugate to any boundary subgroup,  the  normal form of $g$  involves at least two elements $s_i, s_j$ with $1 \leq i < j \leq n-1$.  Since $n \geq 4$, replacing the relation $s_1 \cdots s_n=1$ by a cyclic permutation and by relabeling $s_1,  \cdots, s_n$, we may assume that $1 <i<j\leq n-1$. Then $g$ is hyperbolic in the following malnormal cyclic splitting $G=\<s_1, \cdots, s_i|\>*_{s_1\cdots s_i=s_n^{-1}\cdots s_{i+1}^{-1}}\<s_{i+1}, \cdots, s_n|\>$.  The geometric picture in that case is that the curve representing $g$ intersects at least one simple closed curve which separates the surface into two subsurfaces each of which has at least two boundary components.

Now suppose that $g(S) \geq 1$. Let $c$ be a non null-homotopic simple closed curve represented in a handle of $S$. Let $\Lambda$ be the dual splitting induced by $c$ which is in this case an HNN-extension. If $g$ is hyperbolic in that splitting (in particular if $g$ is a conjugate to a power of the element represented by $c$) then we are done. Otherwise $g$ is elliptic and thus it is conjugate to an element represented by a curve in the surface obtained by cutting along $c$. This new surface $S'$ has genus $g(S)-1$ and two new boundary components and thus we have $2g(S')+b(S')=2g(S)+b(S) \geq 4$. By induction, since $g$ is not conjugate to any  element of a boundary subgroup (also the new two boundaries) of $S'$,  there exists a malnormal cyclic splitting of $S'$ in which $g$ is hyperbolic and the boundary subgroups are elliptic. Since that splitting is compatible with $\Lambda$ we get the required result. 
\end{proof}

\subsection{Homogeneity, algebraic closure, independence}

In this subsection we collect facts about algebraic closure and independence which will be used in the sequel. We start from known facts which we cite for convenient reference and
extend them for our purposes.

Whenever there is more than one group around we may write $\acl_G(A)$ and $\acleq_G(A)$  to denote the algebraic closure of $A$ in the sense of the theory of $G$, and similarly
${A} \Ind^G_{C}{B}$ to describe independence in the theory of $G$.

\begin{proposition}
 \label{prop-spli-acl} Let $G$ be a torsion-free CSA-group. Let $\Lambda$ be a malnormal cyclic splitting of $G$. Then for any nontrivial vertex group $A$ we have $\acl(A)=A$.
\end{proposition}
\begin{proof} A consequence of \cite[Proposition 4.3]{OHV}.
\end{proof}

\begin{theorem} \emph{\cite[Theorem 4.5]{OHV}}\label{theorem-acl}  If $F$ is a free group of finite rank  with nonabelian subgroup $A$, then  $\acl(A)$ coincides with the vertex group containing $A$ in the generalized malnormal (cyclic) $JSJ$-decomposition of $F$ relative to~$A$. \qed
\end{theorem}

\begin{theorem} \emph{\cite{OHV}} \label{thm-finite-g-acl} Let $\Gamma$ be a torsion-free hyperbolic group and $A$ a subgroup of $\Gamma$. Then $\acl(A)$ is finitely generated. \qed
\end{theorem}

\begin{proposition}\label{prop-isol} \emph{\cite[Proposition 5.9]{hom-ould}} Let $F$ be a nonabelian free group of finite rank and $\bar a$ a tuple from $F$ such that the subgroup $A$ generated by $\bar a$ is nonabelian and $F$ is freely  $A$-indecomposable. Then for any tuple $\bar b$ contained in $F$, the type $tp(\bar b/\bar c)$ is isolated.  \qed
\end{proposition}

\begin{theorem} \emph{\cite{hom-ould, perin-homo}} \label{thm1} Let $F$ be a nonabelian free group of finite rank. For any  tuples $\bar a, \bar b \in F^n$ and for any subset $P \subseteq F$,  if $tp^F(\bar a/ P)=tp^F(\bar b/ P)$ then there exists an automorphism   of $F$  fixing $P$ pointwise  and sending $\bar a$ to $\bar b$. \qed
\end{theorem}

One ingredient in the proof of Theorem~\ref{T-main} is a
result of Sela which essentially allows us to work in a free group. In [Sel09], Sela shows that if $\G$ is a torsion-free hyperbolic group which  is not 
elementarily equivalent to a free group, then $\G$ has a \textit{minimal} elementary subgroup, 
denoted by $EC(\G)$, called the \textit{elementary core} of  $\G$ (see  for instance \cite[Section 8]{hom-ould} for some  properties of the elementary core).  
It follows from the definition of the elementary core   that if $\Gamma$ is a nonabelian  torsion-free hyperbolic group then $\Gamma$ is elementarily equivalent to $\Gamma *F_n$ for any free group of rank $n$. Combined with \cite[Theorem 7.2]{sel-free-product} it gives the following  stronger result:

\begin{theorem} \label{thm-equiv-infiniterank}
Any nonabelian torsion-free hyperbolic group $\Gamma$ is elementarily equivalent to
$\Gamma*F$ for any free group $F$.\qed
\end{theorem}

We will be using the following variant of this result:

 \begin{lemma} \label{lem-elem} Let $\G$ be a  torsion-free hyperbolic group not elementarily equivalent to a free group. Let $F$ be a free group  and $C$ a free factor of $F$. Then $EC(\G)*C \preceq \G *F$.  \end{lemma}

 \begin{proof} Suppose first that $F$ has a finite rank. The proof is by induction on the rank $n$ of $C$. If $n=0$ then the result is a consequence of the definition of elementary cores (see \cite{sela09}). Suppose that the result holds for free factors of rank $n$ and set $C=\<\bar c_n, c_{n+1}|\>$.  Let $\phi(\bar x)$ be a formula (without parameters) and $\bar g \in  EC(\G)*C$ such that $\phi(\bar g)$ holds in $EC(\G)*C$.  Then there exists a tuple of words $\bar w(\bar x_n; y)$ such that $\bar g=\bar w(\bar c_n;c_{n+1})$.  By induction $EC(\G)*\<\bar c_n|\>$ is an elementary subgroup of $EC(\G)*C$ and thus by  \cite[Lemma 8.10]{OHV}, the formula $\varphi(\bar w(\bar c_n;y))$ is generic.  Since by induction $EC(\G)*\<\bar c_n|\>$ is an elementary subgroup of $\G*F$ we conclude that there exists $g_1, \cdots, g_p \in EC(\G)*\<\bar c_n|\>$ such that $\G*F=\cup_i g_i X$ where $X=\varphi(\bar w(\bar c_n; \G*F))$.  Hence $c_{n+1} \in g_i X$ for some $i$. There exists an automorphism of $\G*F$  fixing $EC(\G)*\<\bar c_n|\>$ pointwise and sending $c_{n+1}$ to $g_ic_{n+1}$ and thus $\varphi(\bar w(\bar c_n;c_{n+1}))$ holds in $\G*F$ as required. 
 
Suppose now that $F$ has an infinite rank and $C$ has a finite rank. By quantifier elimination and Theorem \ref{thm-equiv-infiniterank}  it is sufficient to show that for any formula $\varphi$ of the form $\forall \exists$ or $\exists \forall$ with parameters from $EC(\G)*C$ if $\G*F \models \varphi$ then $EC(\G)*C \models \varphi$.

Suppose that $\varphi$ is of the form $\forall x \exists y \varphi_0(x;y)$ where $x$ and $y$ are finite tuples and $\varphi_0$ is quantifier-free. If   $\G*F \models \varphi$ then for any $g \in EC(\G)*C$ there exists a free factor $C'$ of $F$ of finite rank such that $C \leq C'$ and a tuple $ a \in \G *C'$ such that $\G*C' \models \varphi_0(g; a)$. By the previous case, $EC(\G)*C$ is an elementary subgroup of $\G*C'$ and thus we get $EC(\G)*C \models \exists y \varphi_0(g; y)$.  Hence we conclude that $EC(\G) *C \models \varphi$.  The case $\varphi$ is of the form $\exists \forall$ can be treated in a similar way.  Finally, the case $F$ and $C$ both have infinite rank follows from the previous cases. \end{proof}

\begin{corollary}\label{l-independence} Suppose  $H=\Gamma*F$ where $\G$ is a torsion-free hyperbolic group and $F$ a free group of finite rank. Then $EC(\G)$ and $F$ are (in $H$) independent over the emptyset.
\end{corollary}
\begin{proof}
If $EC(\G)=1$ the result is clear. So we suppose that $EC(\G)\neq 1$; that is $\G$ is not elementarily equivalent to a free group. Combining Lemma \ref{lem-elem} and  \cite[Lemma 8.10]{OHV},  if $e \in F$ is a primitive element then $tp(e/EC(\G))$ is the unique generic type of $H$ over $EC(\G)$. If $\bar h$ is a generating tuple for $EC(\G)$ we therefore have ${e} \Ind^H{\bar h}$ for
any primitive element $e \in F$. Again combining Lemma \ref{lem-elem} and  \cite[Lemma 8.10]{OHV}  we conclude that  ${\bar e} \Ind^H{\bar h}$ 
for any basis $\bar e$ of $F$. The result now follows.
\end{proof}

In a free group of infinite rank, we obtain the following of indepedent interest:

\begin{lemma}\label{lem-independence}
Suppose that $G$ is a group such that the theory $T_H$ of  $H=G*F$ is simple where $F$ is a free group of infinite rank. Then
\[\indep{G}{}{F}.\]
\end{lemma}
\begin{proof}
Let $\{e_i\colon i<\omega\}$ be part of a basis for $F$. By using Poizat's observation as in \cite{pillay-free},  if $X$ is a definable generic subset of $H$  with parameters from $G$, $X$ contains all but finitely many elements of any basis of $F$. It follows that  the $e_i$ form a Morley
sequence in the sense of $T_H$ over $G$. 
Clearly,  the $e_i$ are indiscernible over any finite subset $A\subset G$.
Therefore (see e.g. \cite[Lemma 7.2.19]{TZ}) \[\indep{A}{}{\{e_i\colon i<\omega\}}\]
and this is enough.
\end{proof}

The following characterization of forking independence over free factors in free groups was  recently proved by
Perin and Sklinos \cite{perin-homo-personal} using \cite[Corollary 2.7]{pillay-free} and Theorem~\ref{thm1}.

\begin{proposition}\rm{\cite{perin-homo-personal}}\label{p-forkingindep}  Let $F$ be a free group of finite rank, $\bar a,\bar b$ be finite tuples from $F$ and $C$ a free factor of $F$. Then \[\indep{\bar a}{C}{\bar b}\]
if and only if 
\[F=A*C*B\mbox{\  with\ \ }\bar a \in A*C\mbox{ and\ \ }\bar b \in C*B.\]
\end{proposition}

We will need slight extensions of the previous results to free products of a \tfh group with a free group:

\begin{proposition} \emph{(Generic  homogeneity)}\label{prop-generichomog}  Let $H=\G*F$ where $\G$ is a torsion-free hyperbolic group (possibly trivial) and $F$ a  free group of finite rank.   Let $\bar a, \bar b, \bar c$ be finite  tuples from $F$. If $tp_H(\bar a/\bar c)=tp_H(\bar b/\bar c)$ then  there exists an automorphism $f \in Aut_{\bar c }(F)$ such that $f(\bar a)=\bar b$. 

In particular,  \[tp_H(\bar a/\bar c)=tp_H(\bar b/\bar c)\]
if and only if \[tp_F(\bar a/\bar c)=tp_F(\bar b/\bar c)\] if and only if  \[tp_H(\bar a/\bar c\bar h)=tp_H(\bar b/\bar c\bar h)\] for any generating set $\bar h$ of
$\G$. 
\end{proposition}

\begin{proof} Let $A$ (resp. $B$) be the subgroup generated by $\bar a$ and  $\bar c$ (resp. $\bar b$ and $\bar c$) and let $E_1$ (resp. $E_2$) be the smallest free factor of $F$ containing $A$ (resp. $B$).   By \cite[Proposition 7.1]{perin-homo}, either there exists an embedding $u : E_1 \rightarrow H$ which fixes $\bar c$ and sends $\bar a$ to $\bar b$, or there exists a noninjective preretraction $r : E_1 \rightarrow H$ with respect to $\Lambda$, the JSJ-decomposition of $E_1$ relative to $A$. If the last case holds then by \cite[Proposition 6.8]{perin-homo}, there exists a noninjective preretraction from $E_1$ to $E_1$ and  by \cite[Proposition 6.7]{perin-homo} we get a subgroup $E'_1\leq E_1$ and a preretraction $r' : E_1 \rightarrow E_1$ such that $(E_1, E'_1, r')$ is a hyperbolic floor. This implies that $F$ has a structure of hyperbolic tower over a proper subgroup contradicting \cite[Proposition 6.5]{perin-homo}.

We conclude that there exists  $u : E_1 \rightarrow H$ which fixes $\bar c$ and sends $\bar a$ to $\bar b$. Symmetrically, there exists $v : E_2 \rightarrow H$ which fixes $\bar c$ and sends $\bar b$ to $\bar a$. Since $E_1$ (resp. $E_2$)  is freely indecomposable relative to $A$ (resp. $B$)  by using Grushko theorem $u : E_1 \rightarrow E_2$ is an isomorphism  sending $\bar a$ to $\bar b$ and fixing $\bar c$  which can be extended to $F$ and also to $H$.   

This shows in particular that $tp_F(\bar a/\bar c)=tp_F(\bar b/\bar c)$. Conversely, if this last property holds then by homogeneity of $F$, there exists an automorphism $f$ of $F$ fixing $\bar c$ and sending $\bar a$ to $\bar b$. Such an automorphism can be extended to an automorphism $\hat f$ of $H$ (fixing $\G$ pointwise) and thus $tp_H(\bar a/\bar c\bar h)=tp_H(\bar b/\bar c\bar h)$ for any generating set $\bar h$ of
$\G$. Then clearly, we also have $tp_H(\bar a/\bar c)=tp_H(\bar b/\bar c)$. \end{proof}

\begin{corollary} Let $H=\G*F$ where $F$ is a free group (of any rank) and $\G$  is a torsion-free hyperbolic group not elementarily equivalent to a free group.  Let $\bar a, \bar b, \bar c$ be finite  tuples from $F$. Then the conclusions of Proposition \ref{prop-generichomog} hold also in this case. 
\end{corollary}

\begin{proof} Let $F_n$ be a free factor of finite rank of $F$ containing the tuples $\bar a, \bar b, \bar c$ and set $H_n=\G*F_n$.  Suppose that $tp_{H}(\bar a/\bar c)=tp_{H}(\bar b/\bar c)$.  By Lemma \ref{lem-elem}, $EC(\G) *F_n \preceq H_n$ and $EC(\G)*F_n \preceq H$ and hence we get $tp_{H_n}(\bar a/\bar c)=tp_{H_n}(\bar b/\bar c)$.  By Proposition \ref{prop-generichomog} there exists an automorphism $f \in Aut_{\bar c }(F_n)$ such that $f(\bar a)=\bar b$ which can be easily extended to $F$. 
\end{proof}

\begin{theorem}\label{t-indep} Let $H=\G*F$ where $F$ is a free group and  $\G$ is a torsion-free hyperbolic group not elementarily equivalent to a free group. For finite tuples $\bar a,\bar b\in F$ and a free factor  $C$ (possibly trivial)  with finite basis $\bar c$  of $F$ we have
$${\bar a} \Ind^H_{\bar c}{\bar b}$$
if and only if 
$${\bar a}\Ind^F_{\bar c}{\bar b}.$$

\end{theorem}

\begin{proof} Let $\bar h$ be
a generating tuple for $EC(\G)$. 
By  Corollary~\ref{l-independence} or Lemma \ref{lem-independence}  we have ${\bar a\bar c}\Ind^H{}{\bar h}$ and ${\bar a \bar b \bar c}\Ind^H{}{\bar h}$. Using transitivity and monotonicity
of forking in stable theories (see e.g. \cite[Corollary 7.2.17]{TZ}) we have
hence  ${\bar a}\Ind^H_{\bar c}{\bar h}$ and ${\bar a}\Ind^H_{\bar b \bar c}{\bar h}$. 

Therefore, again using transitivity and monotonicity
of forking we have \[{\bar a} \Ind^H_{\bar c}{\bar b}\] 
if and only if \[{\bar a} \Ind^H_{\bar c\bar h}{\bar b}.\]

\noindent
Thus it is sufficient to show that 

$${\bar a} \Ind^H_{\bar c \bar h}{\bar b}$$
if and only if 
$${\bar a}\Ind^F_{\bar c}{\bar b}.$$

 The proof is an adaptation of the argument in \cite{perin-homo-personal} by  using the characterisation of forking independence in $F$ given in Proposition ~\ref{p-forkingindep} and Proposition \ref{prop-generichomog}.  Write $F=C*D$ and let $\{d_i: i <\lambda\}$ be a basis of $D$.

Since $EC(\G)*F \preceq \G*F$ we may work in $EC(\G)*F$ and thus without loss of generality we assume that $\G=EC(\G)$. Suppose that ${\bar a} \Ind^H_{\bar h \bar c}{\bar b}$. Let $D'$ be another copy of $D$ with basis $\{d_i': i <\lambda\}$ and consider $H'=\G *F *D'=\G*C*D*D'$. Let $\bar w(\bar x; \bar y)$ be a tuple of words such that $\bar a=\bar w(\bar d, \bar c)$ and consider $\bar a'=w(\bar d', \bar c)$.  Since $H \preceq H'$, we have ${\bar a} \Ind^{H'}_{\bar h \bar c}{\bar b}$. Then $\bar a \Ind^{H'}_{\bar h \bar c}{\bar b}$ and $\bar a' \Ind^{H'}_{\bar h \bar c}{\bar b}$ by Corollary \ref{l-independence} and Lemma \ref{lem-independence}.  

 Since by Lemma~\ref{lem-elem}, $\G*C$ is an elementary subgroup of $H'$, it is algebraically closed in $H'^{eq}$.  It follows that  for any $\bar a \in F$, $tp(\bar a/\bar h \bar c)$ is stationary; that is if $\bar a', \bar a'' , \bar b \in F$, $tp(\bar a'/\bar h \bar c)=tp(\bar a''/\bar h \bar c)=tp(\bar a/\bar h \bar c)$ and  $a' \Ind^H_{\bar h \bar c}{\bar b}$, $a'' \Ind^H_{\bar h \bar c}{\bar b}$ then $tp(\bar a'/\bar b \bar h \bar c)=tp(\bar a''/\bar b \bar h \bar c)$.

By stationarity we get $tp_{H'}(\bar a /\bar h \bar c, \bar b)=tp_{H'}(\bar a'/\bar h \bar c, \bar b)$ and in particular $tp_{H'}(\bar a /\bar c, \bar b)=tp_{H'}(\bar a'/\bar c, \bar b)$. By Proposition~\ref{prop-generichomog}, there exists an automorphism $f$ of  $F*D'$ fixing $\bar c, \bar b$ pointwise and sending $\bar a'$ to $\bar a$.  We have $F*D'=f(D)*C*f(D')$,  $\bar a \in C*f(D')$ and $\bar b \in f(D)*C$.  The conclusion now follows from Grushko's theorem.

For the converse suppose that $F=A*C*B, \hbox{ such that } \bar a \in A*C,  \bar b \in C*B$. Since $\bar a, \bar b, \bar c$ are finite tuples we have a free factor $F_n$ of finite rank with the property $F_n=(A\cap F_n)*C *(B\cap F_n)$ such that $\bar a \in (A\cap F_n)*C,  \bar b \in C*(B \cap F_n)$.  Set $H'=EC(\G)*F_n$. Since $H' \preceq H$, it is sufficient to show that ${\bar a} \Ind^{H'}_{\bar h \bar c}{\bar b}$.  By Lemma \ref{lem-elem} and \cite[Lemma 8.10]{hom-ould}, if $\{e_1, \cdots, e_n\}$ is a basis of $A \cap F_n$ and $\{e_1', \cdots, e_m'\}$ is a basis of $B \cap F_n$, then they are independent realisations of the generic type over $\bar h \bar c$.  The result follows by the same argument as in the proof of Corollary \ref{l-independence}.  \end{proof}

\section{The imaginary algebraic closure}

In this section we study $\acleq(A)$ with respect to the three basic equivalence  relations conjugacy, left (right) cosets of cyclic groups, and double cosets of cyclic groups. Note that the algebraic closure is independent of the model. We  first prove a proposition  of independent interest:

\begin{proposition}\label{prop-autos} Suppose that $G$ is a torsion-free CSA-group, $A$ a nonabelian  subgroup of $G$. Let  $\Lambda=(\mathcal G(V,E), T, \phi)$ be a malnormal  abelian   splitting of $G$ relative to $A$. If $g\in G$ is hyperbolic with respect to this splitting,
there exist $f_n\in Aut_A(G),n \in \Bbb N$  such that the
$f_n(g), n \in \Bbb N$ are pairwise non-conjugate and   $C(f_n(g)) \neq C(f_m(g))$ for $n \neq m$. 
\end{proposition}

We reduce the proof to  the following basic  configurations.

\begin{lemma} \label{lem-config} Let $G$ be a torsion-free CSA-group and $A$ a nontrivial  subgroup of $G$.  Suppose that one of the following cases holds.  

$(i)$ $G=H*_CK$, $C$ is abelian and malnormal, $A \leq H$, $H$ is nonabelian and $g$ is not in a conjugate of $H$ or $K$. 

$(ii)$ $G=\<H, t| C^t=\varphi(C)\>$, $C$ (and $\varphi(C)$)  is abelian and malnormal in $H$, $A \leq H$ and $g$ is not in a conjugate of $H$.

Then there exists infinitely many automorphisms $f_i \in Aut_A(G)$ whose restriction to $H$ is the identity and to  $K$ a  conjugation such that $f_i(g)$ is not conjugate to $f_j(g)$ and    $C(f_i(g)) \neq C(f_j(g))$ for $i \neq j$.\end{lemma} 

\begin{proof} We treat the case $(i)$.  Write $g$ in normal form $g=g_1 \cdots g_r$, with $r \geq 2$, where we may assume that $g$ is cyclically reduced.   If $C=1$ choose $c \in H$ such that $[c,g_i]\neq 1$ for at least one  $g_i$ appearing in the normal form of $g$ with $g_i \in H$. Such a choice is possible since $H$ is nonabelian and CSA.  If $C \neq 1$ we take $c \in C$ nontrivial.   We define the automorphism $f_n$ by being the   identity on $H$ and  conjugation
by $c^n$ on $K$.  We see that $f_n \in Aut_A(G)$. 

 Suppose towards a contradiction that the orbit $\{f_n(g); n \in \Bbb N\}$ is finite up to conjugacy.  Hence there exists $n_0$ and infinitely many $n$ such that each $f_n$ is conjugate to $f_{n_0}$.

Suppose that $C \neq 1$ and thus $c \in C$. In that case, we see that $f_n(g_1) \cdots f_n(g_r)$ is a normal form of $f_n(g)$ which is moreover cyclically reduced for any $n$. Hence by the conjugacy theorem, there exists $d \in C$ such that $f_n(g)$ is conjugate by $d$ to the  product  of a cyclic permutation of the normal form of $f_{n_0}(g)$. Since the number of that cyclic permutations is finite we conclude that there exist $n \neq m$ such that $f_n(g) = f_m(g)^d$, where $d \in C$. 

Suppose that $g_r \in K$. Then $g_rc^{(n-m)}d^{-1}g_r^{-1}\in C$ and thus we must have $c^{n-m}d^{-1}=1$ and $g_{r-1}c^{m-n}g_{r-1}^{-1} \in C$. Since $G$ is torsion-free and $C$ is malnormal, we conclude that $g_{r-1} \in C$; a contradiction. 

Suppose that $g_r \in H$. Then $d=1$ and $g_{r-1}c^{n-m}g_{r-1} ^{-1}\in C$ and the conclusion follows as in the previous case. Hence the set of orbits of $\{f_n(g), n \in \Bbb N\}$ is infinite up to conjugacy. Using a similar argument, we see also that $C(f_n(g)) \neq C(f_m(g))$ for $n \neq m$.

Suppose that $C=1$.  We treat the case $g_r \in L_2$, the other case can be treated in a similar way.  Then we see that $c^nf_n(g)c^{-n}$ is a cyclically reduced conjugate of $f_n(g)$. Proceeding as above, we conclude that there exists $n \neq m$ such that $c^nf_n(g)c^{-n}=c^{m}f_m(g)c^{-m}$. Then using normal forms, we get that $[g_i, c]=1$ for any $g_i$ appearing in the normal form of $g$ with $g_i \in H$; a contradiction. Using a similar argument, we see also that $C(f_n(g)) \neq C(f_m(g))$ for $n \neq m$.

We treat now the case $(ii)$.  Let $c \in G$ be a nontrivial element of $C_{H}(C)$ (if $C=1$, since $A$ is nontrivial, then we take $c$ to be any nontrivial element of $H$) and define for $n \geq 1$, $f_n$     by being the  identity on $H$ and sending $t$ to $c^n t$.  Then $f_n \in Aut_A(G)$. 

Now   $g$ can be written in a normal form $g_0 t^{\epsilon_0} g_1 \cdots g_r t^{\epsilon_r} g_{r+1}$, where $\epsilon_i =\pm 1$.   Since any element is conjugate to a cyclically reduced element, we may assume that $g$ is cyclically reduced and we may take $g_{r+1}=1$.  We have $f_n(g)=g_0 f_n(t^{\epsilon_0})g_1 \cdots g_r f_n(t^{\epsilon_r})$,  $f_n(t^{\epsilon_i})=c^{\epsilon_in} t^{\epsilon_i}$ if $\epsilon_i=1$ and $f_n(t^{\epsilon_i})=\varphi(c)^{\epsilon_i n}t^{\epsilon_i}$ if $\epsilon_i=-1$. Therefore  $f_n(t^{\epsilon_i})g_{i+1} f_n(t^{\epsilon_{i+1}})=c^{\epsilon_i n}t^{\epsilon_i}g_{i+1}\varphi(c)^{\epsilon_{i+1} n}t^{\epsilon_{i+1}}$ if $\epsilon_i=1$ and $\epsilon_{i+1}=-1$. Similarly we have  $f_n(t^{\epsilon_i})g_{i+1} f_n(t^{\epsilon_{i+1}})=\varphi(c)^{\epsilon_i n}t^{\epsilon_i}g_{i+1}c^{\epsilon_{i+1} n}t^{\epsilon_{i+1}}$ if $\epsilon_i=-1$ and $\epsilon_{i+1}=1$. 

If $C \neq 1$ then $g_{i+1}c^{\epsilon_{i+1} n} \not \in C$ and $g_{i+1}\varphi(c)^{\epsilon_{i+1} n} \not \in \varphi(C)$ and thus replacing each $f_n(t^{\epsilon_i})$ by its value we obtain a normal form of $f_n(g)$ which is moreover cyclically reduced.

If $C=1$ then, since $G$ is torsion-free, for all but infinitely many $n$, $ g_{i+1}c^{\epsilon_{i+1} n} \neq 1$ and $g_{i+1}\varphi(c)^{\epsilon_{i+1} n} \neq 1$ and thus replacing each $f_n(t^{\epsilon_i})$ by its value we obtain  as above a normal form of $f_n(g)$ which is moreover cyclically reduced. 

Suppose that the set $\{f_n(g)| n \in \mathbb N\}$ is finite up to conjugacy.  Hence there exist $n_0$ and infinitely many $n$ such that $f_n(g)$ is cyclically reduced and is conjugate to $f_{n_0}(g)$. By the conjugacy theorem, there exists $d \in C \cup \varphi(C)$ such that $f_n(g)$ is conjugate by $d$ to the  product  of a cyclic permutation of the normal form of $f_{n_0}(g)$. Since the number of that cyclic permutations is finite we conclude that there exist $n \neq m$ such that $f_n(g) = f_m(g)^d$, where $d \in C \cup \varphi(C)$.

Since $G$ is a CSA group, either $C$ and $\varphi(C)$ are conjugate in $H$ and in this case  we may assume that $C=\varphi(C)$ and thus $G$ is an extension of a centralizer;  or $C$ and $\varphi(C)$ are conjugately separated, that is $C^h \cap \varphi(C)=1$ for any $h \in H$.  

Suppose first that $C=\varphi(C)\neq 1$.  Since $f_n(g)(f_m(g)^{d})^{-1}=1$, using normal forms we have 
$$
g_r c^{\epsilon_r n}t^{\epsilon_r}d^{-1}t^{-\epsilon_r}c^{-\epsilon_r m}g_r^{-1} \in C
$$
and hence
$$
g_r c^{\epsilon_r (n-m)}d^{-1}g_r^{-1} \in C.
$$

By induction on $0 \leq l \leq r$, we get 
$$
g_l \cdots g_r c^{(\epsilon_l +\cdots + \epsilon_r)(n-m)} d^{-1}g_r^{-1} \cdots g_l^{-1} \in C
$$
and thus, we conclude that
$$
c^{(\epsilon_0+\cdots+\epsilon_r)(n-m)}d^{-1}=d^{-1},
$$
and we get  $c^{(\epsilon_0+\cdots+\epsilon_{r})(n-m)}=1$.   

Suppose that there is no indice $l$ such that $c^{(\epsilon_l +\cdots + \epsilon_r)(n-m)} d^{-1}=1$.  Then it follows from above that $g_i \in C$ for any $i$ and thus $g=g_0 \cdots g_r t^{\epsilon_0+\dots+\epsilon_r}$.  Hence $(\epsilon_0+\cdots+\epsilon_r)\neq 0$ and since $c^{(\epsilon_0+\cdots+\epsilon_{r})(n-m)}=1$ and $G$ is torsion-free we get a contradiction.

Suppose that there is some indice $l$ such that $c^{(\epsilon_l +\cdots + \epsilon_r)(n-m)} d^{-1}=1$.  Then for any $i >l$ or $i<l$ we get $g_i \in C$. Then $g=t^{\epsilon_0+\dots+\epsilon_{l-1}}g_0 \cdots g_rt^{\epsilon_{l}+\dots+\epsilon_{r}}$.  If we suppose that $\epsilon_0+\dots+\epsilon_r=0$ then $g$ will be a conjugate of an element of $H$; a contradiction.  Hence $(\epsilon_0+\cdots+\epsilon_r)\neq 0$ and since $c^{(\epsilon_0+\cdots+\epsilon_{r})(n-m)}=1$ and $G$ is torsion-free we get a contradiction.

Suppose that $C=\varphi(C)=1$. Then $d=1$ and $c^{m-n}=1$; which is a contradiction.  Suppose now that $C_1$ and $C_2$ are conjugately separated.  As in the previous case, we conclude after calculation that $c^{n-m}=1$; which is a contradiction.  

We conclude that the orbit $\{f_n(g)| n \in \mathbb N\}$ is infinite up to conjugacy as required. Using a similar argument with normal forms we get  that $[f_n(g), f_m(g)] \neq 1$  for $ n \neq m$ and thus $C(f_n(g)) \neq C(f_m(g))$ for $n \neq m$. \end{proof}

\smallskip
\noindent {\bf Proof of Proposition \ref{prop-autos}}.

To simplify, identify $G$ with $\pi(\mathcal G(V,E), T)$. 
For an edge $e_i$ outside $T$, let $\mathcal G_i(V,E_i)$ be the graph of groups obtained by deleting $e_i$. Then $G$ is an HNN-extension of the fundamental group $G_i=\pi(\mathcal G_i(V, E_i),T)$ and we can write $G=\<G_i, t| C^t=\varphi(C)\>$ with $A \leq G_i$.

Suppose that for some edge $e_i$ outside $T$ the element $g$ is hyperbolic in the corresponding splitting. Then we conclude by  Lemma \ref{lem-config}.

Now assume that there is no edge in $\mathcal G(V, E)$ outside $T$ such that $g$ is hyperbolic in the corresponding splitting as above. 
Let $L$ be   the fundamental group of the graph of groups $\mathcal G(V, E')$ obtained by deleting all the edges outside the maximal subtree $T$. Since $g$ 
is not hyperbolic in any splitting $G=\<G_i, t| C^t=\varphi(C)\>$, we may 
assume   that $g$ is in  $L$ and  $g$ is hyperbolic in $L$. Thus  we can write $L=L_1*_CL_2$ with $g$ is hyperbolic.  We may suppose without loss of generality that $A \leq L_1$.  By Lemma \ref{lem-config}, there exists infinitely many automorphisms $f_i \in Aut_A(L)$ whose restriction to $L_1$ is the identity and to  $L_2$ a  conjugation such that $f_i(g)$ is not conjugate to $f_j(g)$ and    $C(f_i(g)) \neq C(f_j(g))$ for $i \neq j$.

Since each $f_n$ sends each boundary subgroup to a conjugate of itself, $f_n$ has a natural extension $\hat f_n$  to $G$.  If we suppose that $\{\hat f_n(g); n \in \Bbb N\}$ is finite up to conjugacy (in $G$), then for infinitely many $n$, $f_n(g)$ is conjugate to an element of a vertex group; which is a contradiction. Hence $\{\hat f_n(g); n \in \Bbb N\}$ is infinite up to conjugacy (in $G$) and we see also that $C(f_n(g)) \neq C(f_m(g))$ for $n \neq m$ (in $G$). This ends the proof of the proposition.  \qed

\begin{definition} Let $G$ be a group and $A$ a subgroup of $G$, $c \in G$. We say that $c$ is \textit{malnormaly universally elliptic relative to $A$}  if $c$ is elliptic in any malnormal abelian  splitting of $G$ relative to $A$. 
\end{definition}

\begin{corollary}\label{cor-elliptic} Let $G$ be a torsion-free CSA-group  and $A$ a nonabelian subgroup of $G$.  If $c^G$ or $C(c)$ is in $\acleq(A)$, then $c$ is malnormaly universally elliptic relative to $A$. 
\end{corollary}
\begin{proof}
This follows from Proposition~\ref{prop-autos}.
\end{proof}

We write $\acl^c(\bar a)=\acleq(\bar a)\cap S_{E_0}$, that is $\acl^c(\bar a)$ is the set of conjugacy classes $b^F$ in $\acleq(\bar a)$. For any 
subset $A$ of a group $G$ we also write  $A^c=\{b^G \mid b \in A\}$ for the set of conjugacy classes with representatives in $A$.

In the special case that $G$ is free  we do
have the converse of  Corollary~\ref{cor-elliptic}.
We can formulate the following list of equivalent criteria 
for a conjugacy
class $c^F$ to be contained in the imaginary algebraic closure of a subset
in the free group.

\begin{proposition} \label{prop1} Let $F$ be a  free group of finite rank, $A$ a nonabelian subgroup of $F$  and $c\in F$. The following are equivalent:

$(1)$ $c^F \in \acl^c(A)$. 

$(2)$ There exists finitely many automorphisms $f_1,  \dots,  f_p \in Aut_A(F)$ such that for any $f \in Aut_A(F)$, $f(c)$ is conjugate in $F$ to some $f_i(c)$. 

$(3)$ $c$ is malnormaly universally elliptic  relative to $A$.

$(4)$ In any  generalized cyclic JSJ-decomposition of $F$ relative to $A$,  either $c$ is conjugate to some element of  the elliptic abelian neighborhood of  a rigid vertex group or it is conjugate to an element of a boundary subgroup of a surface type vertex group.

\end{proposition}

\begin{proof}  $(1) \Rightarrow (2)$. Suppose that there exists infinitely many automorphisms $f_i \in Aut_A(F)$ such that $f_i(c)$ is not conjugate to $f_j(c)$ for $i \neq j$.  Then each $f_i$ has an unique  extension $\hat f_i$ to $F^{eq}$  and we get  that $\hat f_i(c^F) \neq \hat f_j(c^F)$ for $i \neq j$. Hence $c^F \not \in \acleq(A)$.

$(2) \Rightarrow (3)$. This follows from Proposition \ref{prop-autos} or Corollary \ref{cor-elliptic}. 

$(3) \Rightarrow (4)$.   We let $\Delta$ be a malnormal generalized cyclic JSJ-decomposition of $F$ relative to $A$. Hence $c$ is elliptic in $\Delta$. Clearly $c$ is not in a conjugate of the free factor of $\Delta$.  If $c$ is in a conjugate of a rigid vertex group, there is nothing to prove. Otherwise $c$ is in a conjugate of a surface type vertex group and without loss of generality we may assume that it is included.   In this case,  by \cite[Proposition 7.6]{jsj-gl} (or by a general version of Lemma \ref{split-surface} and Proposition \ref{prop-autos}) $c$ is in a conjugate of  a boundary subgroup.

$(4) \Rightarrow (2)$. Write $F=F_1*F_2$ where $F_1$ contains $A$ and  freely $A$-indecomposable. Let $\Delta$ be a malnormal  cyclic JSJ-decomposition of $F$ relative to $A$. Suppose that $c$ is in some conjugate of  a rigid vertex group or it is conjugate to an element of a boundary subgroup of a surface type vertex group in a generalized malnormal cyclic JSJ-decomposition of $F$ relative to $A$. 
W.l.o.g, we can assume that it is contained in  a rigid vertex group or it is contained in  a boundary subgroup of a surface type vertex group. Since $Mod_A(F_1)$ has finite index in $Aut_A(F_1)$ (Theorem \ref{mod-index}), there are $f_1, \dots, f_p \in Aut_A(F)$ such that for any $f \in Aut_A(F_1)$ there exists $\sigma \in Mod_A(F_1)$ such that $f=f_i \circ \sigma$ for some $i$. Since $F_1$ is freely $A$-indecomposable by Grushko theorem for any $f \in Aut_A(F)$,  $f_{|F_1} \in Aut_A(F_1)$.   By Lemma \ref{lem-modular},  any $\sigma \in Mod_A(F_1)$ sends $c$ to a conjugate of itself. Hence $f(c)=f_i(c^\alpha)=f_i(c)^{f_i(\alpha)}$ and thus for any $f \in Aut_A(F)$, $f(c)$ is conjugate to some $f_i(c)$.

$(2) \Rightarrow (1)$. Since $\acl(A)$ is finitely generated (Theorem \ref{thm-finite-g-acl}), we may assume that $A$ is finitely generated. Write $F=F_1*F_2$ where $A \leq F_1$ and $F_1$ is freely $A$-indecomposable. Since $F_1$ is an elementary subgroup of $F$, $c$ is a in a conjugate of $F_1$.  We assume that $c=c'^{g}$ where $c' \in F_1$.

By Proposition~\ref{prop-isol} let $\varphi_0(x)$ be a formula isolating the type of $c'$ over $\bar a$ where $\bar a$ is a finite generating tuple of $A$. Let $\varphi(z)$  be the following formula in the language $L^{eq}$
$$
\varphi(z):=\exists x (\hat \varphi_0(x) \wedge \pi(x)=z),
$$
where $\hat \varphi_0$ is the relativisation of $\varphi_0$ to the real sort of $F$ and $\pi$ is the projection from the real sort of $F$ to the sort of the conjugacy classes.  Then $F^{eq} \models \varphi( c^F)$. We claim that $\varphi$ has finitely many realizations, which shows that $c^F \in \acleq(A)$. 

Let $d^F \in F^{eq}$ such that $F^{eq} \models \varphi(d^F)$. Then there exists $\alpha \in F$ such that $F \models \varphi_0(\alpha)$ and $\alpha^F=d^F$. Since $\varphi_0$ isolates the type of $c'$ over $\bar a$ and $F$ is homogeneous (Theorem \ref{thm1}), we conclude that there exists an automorphism $f \in Aut_A(F)$ such that $f(c')=\alpha$.  Hence $f(c)=\alpha^{f(g)}$ and thus $\alpha$ is conjugate to some $f_i(c)$ and thus $d^F=f_i(c)^F$ as required.  \end{proof}

\begin{remark} \label{rem1}Let $F$ be the free group with basis $\{a,b\}$.  The  following example shows that in the previous proposition we cannot remove the assumption that $A$ be nonabelian.   Let $A$ be the subgroup generated by $a$. We claim that  $\acl^c(A)=A^c$. Indeed, let $H=F*\<c|\>$. Then $F$ is an elementary subgroup of $H$ as well as the subgroup $K$ generated by $\<a,c\>$.  Hence $\acl^c(A) \subseteq H^c \cap K^c=A^c$.  By a result of Neilsen (see for instance \cite[Proposition 5.1]{Lynd-schup}) any automorphism of $F$ sends $[a,b]$ to a conjugate of $[a,b]$ or $[a,b]^{-1}$. Hence we see that the implication $(2) \Rightarrow (1)$ is not true in this case.  However if we suppose that $A$ is abelian but not contained in a cyclic free factor  in Proposition \ref{prop1} then the same proof of $(2) \Rightarrow (1)$ works in this case.

\end{remark}

For later reference we also note the following:

\begin{corollary}\label{cor-H*F} Let $H=\G*F$ where $F$ is a free group and  $\G$ is a \tfh group not elementarily equivalent to a free group. Let $\bar a$ be a finite tuple from $F$ generating a nonabelian subgroup. Then any conjugacy class $g^H\in\acl_H^c(\bar a)$   has
a representative $g'$ either in $\G$ or in $F$. If  $g'\in \G$, then in fact ${g'}^H\in \acl^c_H(1)=\acl^c_\G(1)$ and if $g'\in F$, then ${g'}^F\in\acl^c_F(\bar a)$.
\end{corollary}
\begin{proof} Let $A$ to be the subgroup generated by $\bar a$. The first part follows directly from Corollary \ref{cor-elliptic}.
For the second part, just note that $\acleq_H(EC(\G))\cap\acleq_H(F)=\acleq_H(1)$ since
$EC(\G)$ and $F$ are independent by Corollary \ref{l-independence} or Lemma \ref{lem-independence}. Since $EC(\G) \preceq \G$ and $EC(\G) \preceq \G *F$, we get  $\acl^c_H(1)=\acl^c_\G(1)$.  Let $F_n$ be a free factor of $F$ of finite rank containing $A$. Again since $EC(\G)*F_n \preceq H$ any $g^H \in \acl_H^c(\bar a)$ has a representative $g'$ in $EC(\G)*F_n$.  If $g'\in F$ and if we suppose that  ${g'}^{F_n} \not \in\acl^c_{F_n}(A)$ then by Proposition \ref{prop1}, we get infinitely many $f_n \in Aut_A(F_n)$ such the $f_n(g')$ are pairwise non conjugate and each $f_n$ has a natural extension to $H$; which is a contradiction. 
\end{proof}

\begin{remark} \label{rem2}Note that if $\G$ is elementarily equivalent to a free group then $\acl_{\G}^c(1)=1^c$. Indeed, by elementary equivalence it is sufficient to show this when $\G$ is free. In that case we see that any $g^{\G} \in \acl^c_{\G}(1)$ is also an element of $\acl^c_{\G}(a) \cap \acl^c_{\G}(b)$ where $\{a,b\}$ is a part of a basis and we see that $\acl_{\G}^c(a)=\<a\>^c$; which gives the required conclusion. However it may be happen that $\acl_{\G}^c(1)=\G^c$ when $\G$ is not elementarily equivalent to a free group. Indeed let $\G$ be a rigid torsion-free hyperbolic group.  Then  $Out(\G)$ is finite by Paulin's theorem \cite{paulin} and $\G$ is homogeneous and prime by \cite{hom-ould}.  The same method as in the proof of Proposition \ref{prop1} $(2) \Rightarrow (1)$, shows that  for any $g \in \G$, $g^{\G} \in \acl_{\G}^c(1)$.

\end{remark}

 For the equivalence relations $E_{i,m},i=1,2$,  $m \geq 1$ and $E_{3,p,q}$, $p,q \geq1$ given in Theorem~\ref{sela-ima}, we denote the corresponding equivalence classes by $[x]_{i,m},i=1,2$, $[x]_{3,p,q}$ respectively. We start with the following lemma:

\begin{lemma}\label{lemma-centralizer} Let $H=\G*F$ where $\G$ is \tfh (possibly trivial)  and $F$ is a  free group. For a finite tuple $\bar a$ from  $F$  such that $\acl_H(\bar a)=\acl_F(\bar a)$
 and  $c\in H$ the following properties are equivalent: 

$(1)$ $C(c) \in \acleq_H(\bar a)$.

$(2)$ $c\in \acl_H(\bar a)$.
\end{lemma}
\begin{proof} Clearly, $(2)$ implies $(1)$. Let $A$ to be the subgroup generated by $\bar a$. If $A$ is trivial the result is clear.  To prove $(1)$ implies $(2)$
suppose first that $A$ is abelian; so cyclic and generated by $a$. Let $f_n(x)=x^{a^n}$.  If $c \not \in \acl(A)=C(a)$, then $C(f_n(c))\neq C(f_m(c))$ for $n \neq m$.  Hence $C(c) \notin\acleq_H(A)$. 

Suppose now that $A$ is nonabelian. Let $F_p$ be a free factor of finite rank of $F$ containing $A$ and which is freely $A$-indecomposable and set $F=F_p*D$.  Let $\Delta$ be the  malnormal cyclic JSJ-decomposition of $F_p$ relative to $A$. Extend this to a decomposition $\Delta'$ of $H=(\G*D)*F_p$. Then by Theorem~\ref{theorem-acl} and Proposition~\ref{prop-spli-acl} the vertex group in $\Delta'$ containing $A$ is $\acl_{F_n}(A)=\acl_{F}(A)=\acl_H(A)$ and by Corollary~\ref{cor-elliptic} $c$ is elliptic with respect to $\Delta'$. 

Suppose that $c$ is in a conjugate of $\G*D$ and set $c=c_0^a$ with $c_0 \in \G*D$.  If $a\neq 1$,  let $f_n$ denote the automorphism of $H$ given by conjugation by $a^n$ on $\G*D$ and the identity on $F_p$. Then clearly $c^{a^n}$ and $c^{a^m}$ do not centralize each other for $n\neq m$  showing
that $C(c)\notin\acleq_H(A)$. Hence $a=1$ and $c \in \G*D$. Again let $a\in F_p$ and let $f_n$ denote the automorphism of $H$ given by conjugation by $a^n$ on $\G*D$ and the identity on $F_p$. As above  $c^{a^n}$ and $c^{a^m}$ do not centralize each other for $n\neq m$  showing
that $C(c)\notin\acleq_H(A)$.

Therefore we may assume that $c$ is in a conjugate of $F_p$. Proceeding as above we get $c \in F_p$.  Suppose now that $c$ is not in a conjugate of $\acl(A)$. Write $\Delta= (\mathcal G(V,E), T, \phi)$ and let $L$ to be the fundamental group of the graph of groups obtained by deleting the edges which are outside $T$.  Hence $c$ is in a conjugate of $L$ and without loss of generality we may assume that $c \in L$. Let $e_1, \cdots, e_q$ be the edges adjacent to $acl(A)$ and let $C_i$ be the edge group corresponding to $e_i$. Hence we can write $L=L_{i1}*_{C_i}L_{i2}$ with $\acl(A) \leq L_{i1}$. Since $c$ is elliptic and $c$ is not in a conjugate of $\acl(A)$, we get, without loss of generality  that $c \in L_{i2}$ for some $i$.  If $c \in C_i$ then $c$ is in a conjugate of $\acl(A)$ contrary to our hypothesis. So $c \not \in C_i$. Proceeding as in the proof of Proposition \ref{prop-autos},  we take  infinitely many 
Dehn twists $f_n \in Aut_A(F_n)$ around $C_i$ if $C_i \neq 1$ and a conjugation by a nontrivial  element of  $L_{i1}$ if $C_i=1$ and thus we find  $C(f_n(c)) \neq C(f_m(c))$ for $n \neq m$. Now each $f_n$ has a standard extension $\hat f_n \in Aut_A(H)$ and we see also that $C(\hat f_n(c)) \neq C(\hat f_m(c))$ for $n \neq m$. It follows that $C(c) \not \in \acleq(A)$; a contradiction. 

Hence $c$ is in a conjugate of $\acl(A)$. Suppose that $c =d^\alpha$ with $d \in \acl(A)$ and $\alpha \in F_p\setminus\acl(A)$. Then we find infinitely many automorphisms $f_i \in Aut_A(F_p)$ such that $f_n(\alpha) \neq f_m(\alpha)$ for $n \neq m$. Clearly, these $f_i$ extend to $H$. Hence  $C(f_n(c))=C(f_n(d))^{f_n(\alpha)}\neq C(f_m(c))$ for infinitely many $n,m$  and thus $C(c) \not \in \acleq(A)$; a contradiction.  We conclude that $c \in \acl(A)$. \end{proof}

\begin{proposition}\label{prop-imag-acl}Let $H=\G*F$ where $\G$ is \tfh (possibly trivial)  and $F$ is a  free group. For a finite tuple  $\bar a$ of $F$ such that $\acl_H(\bar a)=\acl_F(\bar a)$ and   $c,d,e\in H$ the following properties are equivalent:

$(1)$ $[(c,d)]_{1,m} \in \acleq_H(\bar a)$. 

$(2)$ $c,d \in \acl_H(\bar a)$. 

$(3)$ $[(c,d)]_{2,m} \in \acleq_H(\bar a)$. 

Similarly we have
 $[(c,d,e)]_{3,p,q} \in \acleq_H(\bar a)$
 if and only if $c,d,e \in \acl_H(\bar a)$.

\end{proposition}

\begin{proof} 

Let $A$ to be the subgroup generated by $\bar a$. The implications $(2) \Rightarrow (1)$ and  $(2) \Rightarrow (3)$ are clear. By symmetry it suffices to prove
$(1) \Rightarrow (2)$. By Lemma~\ref{lemma-centralizer}, we have $c\in \acl_H(A)$.

If $A$ is abelian and generated by $a$, let $f_i(x)=x^{a^i}$ for $i \in \mathbb N$. If $d \not \in \acl_H(A)=C(a)$, then $f_i(d) \not \in f_j(d)C(a)$ for $i \neq j$.  Thus  for $i \neq j$, $(f_i(c), f_i(d))$ is not equivalent to  $(f_j(c), f_j(d))$  relative to the equivalence relation $E_{1,1}$ and also relative to $E_{1,m}$ for all $m \geq 1$.  Hence  $ [c, d]_{1,m} \not \in \acleq(A)$; a contradiction. Therefore $d \in C(a)=\acl(A)$. 

If $A$ is nonabelian and $d  \in H\setminus\acl(A)$, proceeding as in the proof of Lemma \ref{lemma-centralizer},  we can again find infinitely many automorphisms $f_i \in Aut_A(F)$ such that $f_i(d) \not \in f_j(d) C(c)$ for $i \neq j$. Clearly, these $f_i$ extend to $H$ and so $[(c,d)]_{1,m} \notin \acleq(A)$. \end{proof}

\bigskip
Recall that we write $\acl^c(\bar a)=\acleq(\bar a)\cap S_{E_0}$. For any 
subset $A$ of a group $G$ we also write  $A^c=\{b^G \mid b \in A\}$ for the set of conjugacy classes with representatives in $A$.

\begin{corollary}\label{cor-acleq} Let $H=\Gamma*F$ where $\G$ is \tfh (possibly trivial)  and $F$ is a nonabelian free group. For finite tuples $\bar a,\bar b,\bar c\in F$  we have
$$
\acleq_H(\bar a) \cap \acleq_H(\bar b)=\acleq_H(\bar c)
$$
if and only if 
$$
\acl_H^c(\bar a) \cap \acl_H^c(\bar b)=\acl_H^c(\bar c) \leqno (1)
$$ 
and
$$
\acl_H (\bar a) \cap \acl_H (\bar b) =\acl_H(\bar c). \leqno (2)
$$

\end{corollary}

\begin{proof} By Theorem \ref{thm-equiv-infiniterank} we have $\G*F_2$ is elementarily equivalent to $H$ and thus we can apply Theorem \ref{sela-ima}. 

One direction is clear. For the other direction, 
by Theorem \ref{sela-ima} and Remark \ref{lem-link} it suffices to show
that 

$$(\acleq(\bar a) \cap F_{\mathcal E}) \cap (\acleq(\bar b) \cap F_{\mathcal E})=\acleq(\bar c) \cap F_{\mathcal E}$$

where $\mathcal E$ is the set of equivalence relations given in Theorem \ref{sela-ima}. For $E_0$ this is assumption $(1)$, for $E_{1,m},E_{2,m},E_{3,p,q}$ this
follows from $(2)$ and Proposition~\ref{prop-imag-acl}.
\end{proof}

\begin{definition} Let $G$ be a group and $\bar a $ a tuple from $G$. We say that $\bar a $ \rm{represents conjugacy} (in $G$)  if $\acl_G^c(\bar a)=\acl_G(\bar a)^c$. 

\end{definition}

To verify the properties of an ample sequence in a free factor of a
\tfh group it now suffices to restrict to this free factor:

\begin{lemma}\label{lemma-freefactor}
Let $H=\G*F$ where  $F$  is a  free group and $\G$  is \tfh not elementarily equivalent to a free group. 
For finite tuples $\bar a,\bar b,\bar c\in F$ generating nonabelian subgroups, representing conjugacy  in $F$  and such that $\acl_H(\bar a)=\acl_F(\bar a), \acl_H(\bar b)=\acl_F(\bar b), \acl_H(\bar c)=\acl_F(\bar c)$, we have   

\[\acleq_F(\bar a)\cap\acleq_F(\bar b)=\acleq_F(\bar c)\]

if and only if 

\[\acleq_H(\bar a)\cap\acleq_H(\bar b)=\acleq_H(\bar c).\]

\end{lemma}
\begin{proof}
By Corollary~\ref{cor-acleq} and the assumption on $\bar a,\bar b$ and $\bar c$
it suffices to verify that 
\[
\acl_F^c(\bar a) \cap \acl_F^c(\bar b)=\acl_F^c(\bar c)
\]
if and only if 
\[
\acl_H^c(\bar a) \cap \acl_H^c(\bar b)=\acl_H^c(\bar c). 
\]

But this follows from Corollary~\ref{cor-H*F} and the assumption that the considered tuples represent conjugacy:
for any nonabelian subgroup $A\leq  F$ 
the conjugacy classes in $\acl_H^c(A)$ have representatives either in $\acl_F^c(A)$ or in  $\acl_H^c(1)$ and since $\bar a$ say represent conjugacy in $F$ we have in fact that $g^H \in \acl_H^c(\bar a)$ if and only if either $g^H \in \acl_H^c(1)$ or $g$ has a representative $g' \in F$ such that $g'^F\in \acl_F^c(\bar a)$. 
\end{proof}

\begin{corollary}\label{cor-reduction}
Let $H=\G*F$ where  $F$  is a  free group and $\G$  is \tfh not elementarily equivalent to a free group. Suppose that $a_0,\ldots, a_n$ are finite tuples in $F$, each $a_i$ generating a nonabelian free factor of $F$  and witnessing
that $F$ is $n$-ample and such that for $0\leq i\leq k$ we have
\[\acl_F(a_0,\ldots a_i,a_k)=\acl_H(a_0,\ldots a_i,a_k),\]
\[\acl_F^c(a_0,\ldots a_i,a_k)=\acl_F(a_0,\ldots a_i,a_k)^c.\]
Then $a_0,\ldots, a_n$ witness the fact that $Th(\G)$ is $n$-ample.
\end{corollary}

\begin{proof}
This follows from Lemma~\ref{lemma-freefactor} and Theorem~\ref{t-indep}. 
\end{proof}

\section{The construction in the free group}

In this section we will be working exclusively in nonabelian free groups and therefore all notions of algebraic closure and independence refer to the theory $T_{fg}$. Our
main objective here is to construct sequences witnessing the ampleness. 
Corollary~\ref{cor-reduction} then allows us to transfer the results in Section~5
to \tfh groups to obtain our main theorem.

\bigskip
\noindent
Let 
$$
H_i=\<c_i, d_i, a_i, b_i \mid c_id_i[a_i, b_i]=1\>,
$$
that  is  $H_i$ is the fundamental group of an orientable surface with $2$ boundary components and genus~$1$, where $c_i$ and $d_i$ are the generators of  boundary subgroups. Note that $H_i$ is a free group of rank $3$ with bases $a_i,b_i,c_i$ or $a_i,b_i,d_i$. Let
$$
P_n=H_0*H_1*\cdots *H_{n-1}*H_n,
$$
and 
$$
G_0=P_0=H_0, \; \; G_n=\<P_n, t_i, 0 \leq i \leq n-1\mid d_i^{t_i}=c_{i+1}\> \hbox{ for }n \geq 1. 
$$

\begin{remark}
Note that for any $k< n$ we have
$$
G_n=\<G_k*H_{k+1}*\ldots *H_n, t_j,k\leq j\leq n-1\mid d_j^{t_j}=c_{j+1}\>.
$$
\end{remark}

One of the principal properties that we will use is that gluing together surfaces on boundary subgroups gives new surfaces. For $i\geq 0$, let  $\overline h_i=(a_i,b_i,c_i)$ be the given  basis of $H_i$. We are going to  show that the sequence $\overline h_0,\overline h_2,\ldots, \overline h_{2n}$ is a witness for the $n$-ample property in $G_{2n}$.  The proof is divided into a sequence of lemmas.

\begin{lemma}\label{l-freefactor} $\;$

$(1)$  For $0 \leq i \leq n$, $G_i$ is a free factor of $G_n$ and
$G_n$ is a free group of rank $3(n+1)$. 

$(2)$ For each $0 \leq i \leq n$, $H_i$ is a free factor of $G_n$. 

\end{lemma}

\begin{proof} $\;$
$(1)$ The proof is by induction on $n$. For $n=0$, we already noted that $G_0=H_0$ is a free group of rank $3$. For the induction step it suffices to show that $G_n$ is a free factor of $G_{n+1}$ with a complement which is free of rank $3$. We have 
$$
G_{n+1}=(G_n*\<t_n\mid \>)*_{d_n^{t_n}=c_{n+1}}\<c_{n+1}, a_{n+1}, b_{n+1}\mid \>,
$$
and since $c_{n+1}$ is primitive in $\<c_{n+1}, a_{n+1}, b_{n+1}\mid \>$, the free group generated by $t_n,a_{n+1},b_{n+1}$ is a free factor in $G_{n+1}$.
This proves the claim.

$(2)$ In view of $(1)$ it suffices to show by induction on $i$ that $H_i$ is a free factor of $G_i$ for $0\leq i\leq n$. For $i=0$, there is nothing to prove. 
For $i+1$, we have  as above
$$
G_{i+1}=(G_i*\<t_i\mid \>)*_{d_i^{t_i}=c_{i+1}}\<c_{i+1}, a_{i+1}, b_{i+1}\mid \>,
$$
and since by induction $H_i$ is a free factor of $G_i$, $d_i$ is primitive in $G_i$. In particular, $d_i^{t_i}$ is primitive in $(G_i*\<t_i\mid \>)$ and we conclude that $H_{i+1}$ is a free factor of $G_{i+1}$, as required. 
\end{proof}

\begin{lemma}\label{lemma-surface}  For $n=2k\geq 2$, we can write
$$
G_n=S*\<t_0\mid \>*\ldots *\< t_{n-1}\mid \>. \leqno (1)
$$
for a surface group $S$ with
$$
S=\<c_0, d_n', a_0, b_0, a'_1, b'_1,\ldots  a'_n, b'_n\mid c_0d_n'[a_n',n_1']\ldots [a_1',b_1'][a_0, b_0]=1\>
$$
\noindent
where $d_n'$ is conjugate to $d_n$.

\end{lemma}

\begin{proof} 
We have 
$$
G_n=\<a_0, b_0\mid \>*(\<t_0\mid \>*\<a_1, b_1\mid \>*\ldots *\<t_{n-1}\mid \>)*\<d_n, a_n, b_n\mid \>, 
$$
where

$$
c_n=[b_n, a_n]d_n^{-1},  \; d_{n-1}=t_{n-1}c_nt_{n-1}^{-1},  \; c_{n-1}=[b_{n-1}, a_{n-1}]d_{n-1}^{-1},\ldots ,\; d_0=t_0c_1t_0^{-1}, \; c_0=[b_0, a_0]d_0^{-1}.
$$

\noindent
Replacing successively and setting $s_i^{-1}=t_0\ldots t_i$ we obtain:

$$
c_0=[b_0,a_0][b_2,a_2]^{s_1}[b_4,a_4]^{s_3}\ldots [b_n,a_n]^{s_{n-1}} (d_n^{-1})^{s_{n-1}}[a_{n-1},b_{n-1}]^{s_{n-2}}\ldots [a_1,b_1]^{s_0}.
$$

\noindent
For $0<i\leq n$ put
$$
a'_i=a_1^{s_{i-1}}, b'_i=b_1^{s_{i-1}}, \; \mbox{ and }\; d_n'=d_2^{s_{n-1}}.
$$

\noindent
Then 

$$
c_0=[b_0,a_0][b'_2,a'_2][b'_4,a'_4]\ldots [b'_n,a'_n] (d'_n)^{-1}[a'_{n-1},b'_{n-1}]\ldots [a'_1,b'_1].
$$

\noindent
Finally for $i=1,\ldots k$, put 
$$
a''_{2i-1}=a_{2i-1}^{d'_n},\; \mbox{ and }\; b''_{2i-1}=b_{2i-1}^{d'_n}.
$$

\medskip
\noindent
Then 
$$
G_n=\<a_0, b_0\mid \>*\<t_0\mid \>*\<a''_1, b''_1\mid \>*\<t_1\mid \>*
\<a'_2, b'_2\mid \>*\<t_2\mid \>*\ldots *(\<t_{n-1}\mid \>*
\<d_n', a'_n, b'_n\mid \>,
$$

and

$$
c_0=[b_0,a_0][b'_2,a'_2][b'_4,a'_4]\ldots [b'_n,a'_n] [a''_{n-1},b''_{n-1}]\ldots [a''_1,b''_1](d'_n)^{-1}.
$$

\noindent
With
$$
S=\<c_0, d_n', a_0, b_0, a''_1, b''_1, a'_2, b'_2,\ldots a'_n,b'_n\mid c_0d_n'[b''_1,a''_1][b''_3,a''_3]\ldots [a_n',b_n']\ldots [a_2,b_2][a_0, b_0]=1\>
$$
we have
$$
G_n=S*\<t_0\mid \>*\ldots *\< t_{n-1}\mid \>. \leqno (1)
$$

 \end{proof}

As $S$ is the fundamental group of an orientable surface with 
genus~$\geq 1$ and two boundary subgroups generated by $c_0$ and $d_n'$
we may apply Lemma~\ref{split-surface} to $S$ and obtain the following
corollary:

\begin{corollary}\label{cor-splitting}
Suppose $n=2k\geq 2$ and $g\in G_n\setminus\{1\}$ is not conjugate to a power of $c_0$ or of $d_n$. Then there exists a malnormal cyclic  splitting of $G_n$ such that $c_0$ and $d_n$ are elliptic and $g$ is hyperbolic.
\end{corollary}

\begin{proof} Write $g=g_1 \cdots g_k$ in normal form with respect to the splitting appearing in Lemma~\ref{lemma-surface}.  W.l.o.g, we may assume that $g$ is cyclically reduced. If $k \geq 2$, then $g$ is hyperbolic in the given malnormal splitting. Hence we may assume that $g \in S$ or $g \in \<t_i\mid\>$ for some $i\leq n-1$.

If $g \in S$, by Lemma~\ref{split-surface}, there exists a malnormal cyclic splitting $S$  in which $g$ is hyperbolic and $c_0$ and $d_n$ are elliptic. This yields a refinement of the cyclic splitting of $G_n$ in Lemma~\ref{lemma-surface} in which $g$ is hyperbolic.

Next suppose that $g \in \<t_0\mid \>$. For any $1\leq i<n$ we can write 
$\<t_0, t_i\mid \>=\<\<t_i, t_0t_it_0^{-1}\mid \>, t_0\mid (t_0t_it_0^{-1})^{t_0}=t_i\>$ which is a cyclic splitting $\Delta$ in which $t_0$ is hyperbolic. By replacing  $\<t_0\mid \>*\< t_i\mid \>$ by $\Delta$
in the splitting given in Lemma~\ref{lemma-surface} we obtain a cyclic splitting of $G_n$ in which $g$ is hyperbolic. If $g \in \<t_i\mid \>$ for $1\leq i<n$ the same proof works and we are done.
\end{proof}

Recall that for $i\geq 0$, $\overline h_i=(a_i,b_i,c_i)$. Having diposed  by some  needed  properties of $G_n$ in the previous lemmas, we  are now ready  to  show that the sequence $\overline h_0,\overline h_2,\ldots, \overline h_{2n}$ satisfies  conditions of Definition \ref{def-ample}.  Since $G_i$ is a free factor of $G_k$ for $0 \leq i \leq k$ which  implies that $G_i$ is an elementary subgroup of $G_k$, in computing the algberaic closure -- as well as the imganiary algebraic closure -- of tuples of $G_i$ it is enough to work in $G_i$.

\begin{lemma}\label{lemma-acl} For $0\leq i<k$ we have 
 \[G_{2i}\cap \acl(\overline h_0, \overline h_2, \dots, \overline h_{2i}, \overline h_{2k})=\acl(\overline h_0, \overline h_2, \dots, \overline h_{2i})
\]

and 
 \[G_{2i}^c\cap \acl^c(\overline h_0, \overline h_2, \dots, \overline h_{2i}, \overline h_{2k})=\acl^c(\overline h_0, \overline h_2, \dots, \overline h_{2i}).
\]

\end{lemma}

\begin{proof}
 Since $G_{2k}$ is a free factor we may work in $G_{2k}$. Let $g \in G_{2i}\cap \acl(\overline h_0, \overline h_2, \dots, \overline h_{2i}, \overline h_{2k})$ and suppose towards a contradiction that $g\in  G_{2i}\setminus \acl(\overline h_0, \overline h_2, \dots, \overline h_{2i})$.  
Then there exist infinitely many automorphisms $f_p \in Aut(G_{2i})$ which fix $\overline h_0, \overline h_2, \dots, \overline h_{2i}$ such that $f_p(g) \neq f_q(g)$ for $p \neq q$.  Recall that 
\[G_{2k}=\<G_{2i}* H_{2i+1}*\ldots *H_{2k-1}*H_{2k},t_j, 2i\leq j\leq 2k-1\mid d_j^{t_j}=c_{j+1} \>. 
\]

We extend each $f_p$ to $G_{2i}*H_{2i+1}*\ldots *H_{2k-1}*H_{2k}$ by the identity on the factor $H_{2i+1}*\ldots *H_{2k-1}*H_{2k}$.  Hence each $f_p$ has a natural extension to $G_{2k}$ that we denote by $\hat f_p$ and such that $\hat f_p(g) \neq \hat f_q(g)$ for $p\neq q$. We see that each $f_p$ fixes $\overline h_0, \overline h_2, \dots, \overline h_{2i}, \overline h_{2k}$ and $\hat f_p(g) \neq \hat f_q(g)$  for $p \neq q$.  Therefore $g \not \in \acl(\overline h_0, \overline h_2, \dots, \overline h_{2i}, \overline h_{2k})$; which is a contradiction.

Similarly if $g^{G_{2i}}\in  G_{2i}^c\setminus \acl^c(\overline h_0, \overline h_2, \dots, \overline h_{2i})$ there exist infinitely many automorphisms $f_p \in Aut(G_{2i})$ which fix $\overline h_0, \overline h_2, \dots, \overline h_{2i}$ such that $f_p(g)$ and $f_q(g)$ are not conjugate for $p \neq q$.  We extend  these $f_p$ to $G_{2k}$ as in the previous paragraph.  Clearly $\hat f_p(g)$ and $\hat f_q(g)$ are not conjugate in $G_{2i}* H_{2i+1}*\ldots *H_{2k-1}*H_{2k}$.   Suppose towards a contradiction that $\{\hat f_p(g) \mid p \in \mathbb N\}$ is finite up to conjugacy in $G_{2k}$. By applying  \cite[Lemma 3.1]{hom-ould} it follows that for $p \neq q$ the pair  $(f_p(g), f_q(g))$ is conjugate in $G_{2i}* H_{2i+1}*\ldots *H_{2k-1}*H_{2k}$ to a  one of the pairs $(d_j^r, c_{j+1}^r)$ for  $2i\leq j\leq 2k-1$ and $r \in \mathbb Z$; this is clearly a contradiction.
By Proposition \ref{prop1}  $g^{G_{2k}} \not \in \acl^c(\overline h_0, \overline h_2, \dots, \overline h_{2i}, \overline h_{2k})$; which is a contradiction.
\end{proof}

\begin{lemma}\label{l-aclH} For $0\leq i\leq k$ we have 
$$
\acl(\overline h_0, \overline h_2, \dots, \overline h_{2i}, \overline h_{2k}) = H_0*H_2*\dots *H_{2i} *H_{2k}  \leqno (i)
$$
and
$$
\acl^c(\overline h_0, \overline h_2, \dots, \overline h_{2i}, \overline h_{2k})=\acl(\overline h_0, \overline h_2, \dots, \overline h_{2i}, \overline h_{2k})^c \leqno (ii)
$$
\end{lemma}

\begin{proof}   We first prove $(i)$. By rewriting the splitting 
\[G_{2k}=\<H_0*H_1*\cdots *H_{2k-1}*H_{2k},  t_i, 0 \leq i \leq 2k-1\mid d_i^{t_i}=c_{i+1}\>,\]
as 
\[G_{2k}=\<(H_0*H_2*\dots *H_{2i} *H_{2k})*K,  t_i, 0 \leq i \leq 2k-1\mid d_i^{t_i}=c_{i+1}\>,\]
where $K$ is the free product of the remaining of the $H_i$, we get a malnormal cyclic splitting in which $H_0*H_2*\dots *H_{2i} *H_{2k}$ is a vertex group. Hence by  Proposition \ref{prop-spli-acl}
\[\acl(\overline h_0, \overline h_2, \dots, \overline h_{2i}, \overline h_{2k}) = H_0*H_2*\dots *H_{2i} *H_{2k}\]
as required. 

We prove  $(ii)$. We assume inductively that
\[
\acl^c(\overline h_0, \overline h_2, \dots, \overline h_{2i}, \overline h_{2k})=\acl(\overline h_0, \overline h_2, \dots, \overline h_{2i}, \overline h_{2k})^c.  
\]

Let  $c^{G_{2k}} \in \acl^c(\overline h_0, \overline h_2, \dots, \overline h_{2i}, \overline h_{2k})$ and consider the malnormal cyclic splitting of $G_{2k}$ with vertex groups $K= (G_{2i}*H_{2k})$ and $L=\<H_{2i+1}*\ldots *H_{2k-1},t_j, 2i+1\leq j\leq 2(k-1)\mid d_j^{t_j}=c_{j+1} \>$
given by
\[G_{2k}=\< K*L,  t_j, j=2i, 2k-1\mid d_j^{t_j}=c_{j+1} \>.  
\] 
By Corollary~\ref{cor-elliptic}, $c$ is in a conjugate of $K$ or of $L$. 
Suppose first that $c$ is in a conjugate of $L$ and thus w.l.o.g, we may assume that $c \in L$. Note that $L\cong G_{2m}$ with $m=k-i-1$. By Corollary~\ref{cor-splitting} when $2m \geq 2$ and Lemma \ref{split-surface} when $2m=0$, if $c$ is not conjugate to an element of $\<c_{2i+1}\>$ or   $\<d_{2k-1}\>$ then $L$ has a malnormal cyclic splitting such that $c_{2i+1}$ and $d_{2k-1}$  are elliptic and $c$ is hyperbolic.  Hence we get a refinement of the previous splitting of $G_{2k}$ where $K$ is still a vertex group and such that $c$ is hyperbolic, contradicting  Corollary~\ref{cor-elliptic}. 

Therefore $c$ is either conjugate to an element of $\<c_{2i+1}\>$ or  of $\<d_{2k-1}\>$. But since  $d_{2k-1}^{t_{2k-1}}=c_{2k}, d_{2i}^{t_{2i}}=c_{2i+1}$ we conclude that either $c^{G_{2k}} \in \acl(\bar h_{2i})^c=H_{2i}^c$ or $c^{G_{2k}} \in \acl(\bar h_{2k})^c=H_{2k}^c$ and hence
$c^{G_{2k}} \in  \acl(\overline h_0, \overline h_2, \dots, \overline h_{2i}, \overline h_{2k})^c $ as required.

Suppose now that $c$ is in a conjugate of $K$ and thus w.l.o.g, we may assume that $c \in K$.  Write $c=g_1\cdots g_m$ in normal form with respect to the free product structure of $K=G_{2i}*H_{2k}$.  Since any element is conjugate to a cyclically reduced one, w.l.o.g, we may assume that $c$ is cyclically reduced. 
If $m=1$, then $c\in H_{2k}$ or $c\in G_{2i}$ and hence $c^{G_{2i}}\in\acl(h_{2k})^c$
or $c^{G_{2i}} \in  \acl(\overline h_0, \overline h_2, \dots, \overline h_{2i})^c $ by
Lemma~\ref{lemma-acl} and induction hypothesis. If $m>1$, we claim that for any $1\leq l\leq m$ if $g_l \in G_{2i}$ then $g_l \in \acl(\overline h_0, \overline h_2, \dots, \overline h_{2i})$ and so $c \in \acl(\overline h_0, \overline h_2, \dots, \overline h_{2i}, \overline h_{2k})$.  Suppose towards a contradiction that $g_l \in  G_{2i}\setminus \acl(\overline h_0, \overline h_2, \dots, \overline h_{2i})$ for some $1\leq l\leq m$.  Then proceeding as in the proof of  Lemma~\ref{lemma-acl} there exist infinitely many automorphisms $f_p \in Aut(K)$ fixing $\overline h_0, \overline h_2, \dots, \overline h_{2i}, \overline h_{2k}$ and with $f_p(c)\neq f_q(c)$ for $p\neq q$.  Each $f_p$ has a natural extension to $G_{2k}$ that we denote by $\hat f_p$ and such that $\hat f_p(c) \neq \hat f_q(c)$ for $p\neq q$.

Suppose that the set $\{f_p(c) \mid p \in \mathbb N\}$ is finite up to conjugacy in $K$. Hence there exists an infinite set $I \subseteq \mathbb N$ and $p_0$ such that $f_p(c)$ is  conjugate to $f_{p_0}(c)$ for every $p \in I$. We see that $f_p(c)=f_p(g_1) \cdots f_p(g_k)$ and $(f_p(g_1), \dots, f_p(g_k))$ is a normal form and $f_p(c)=f_p(g_1) \cdots f_p(g_k)$ is cyclically reduced. Therefore $(f_p(g_1), \dots, f_p(g_k))$ is a cyclic permutation of $(f_{p_0}(g_1), \dots, f_{p_0}(g_k))$. Since the number of such cyclic permutations is finite, we conclude that there exists $p \neq q \in I$ such that $f_p(c)=f_q(c)$ and thus $f_p(g_l)=f_q(g_l)$, a contradiction.

Hence the set $\{f_p(c) \mid p \in \mathbb N\}$ is infinite up to conjugacy in $K$.  Suppose towards a contradiction that $\{\hat f_p(c) \mid p \in \mathbb N\}$ is finite up to conjugacy in $G_{2k}$. Since $\{\hat f_p(c) \mid p \in \mathbb N\}$  is infinite up to conjugacy in $K$, we conclude that there exists $p\neq q$ such that $f_p(g)$ is conjugate to $f_q(c)$ in $G_{2k}$ and  $f_p(c)$ is not conjugate to $f_q(c)$ in $K*L$. By applying \cite[Lemma 3.1]{hom-ould} it follows  that the pair  $(f_p(c), f_q(c))$ is conjugate in  $K*L$ to a  one of the pairs  $(d_{2i}^r, c_{2i+1}^r), (d_{2k-1}^r, c_{2k}^r)$ for some $r \in \mathbb Z$; this is clearly a contradiction as $f_p(c), f_q(c) \in K$. Hence the set $\{\hat f_p(c) \mid p \in \mathbb N\}$ is infinite up to conjugacy in $G_{2k}$ and by Proposition \ref{prop1}  $c^{G_{2k}} \not \in \acl^c(\overline h_0, \overline h_2, \dots, \overline h_{2i}, \overline h_{2k})$; which is a contradiction.
\end{proof}

\begin{lemma} We have 

\[
\acl (\bar h_{0}) \cap \acl (\bar h_{2})=\acl(\emptyset) 
\] 
and
\[
\acl^c(\bar h_{0}) \cap \acl^c(\bar h_{2})=\acl^c(\emptyset). 
\]

\end{lemma}

\begin{proof} As $G_2$ is a free factor, we work in $F=G_2$.  Let $a \in \acl (\bar h_{0}) \cap \acl (\bar h_{2})$. Since $\acl(\bar h_0)=H_0$, $\acl(\bar  h_2)=H_2$ and $a \in \acl(\bar h_0, \bar h_2)=H_0*H_2$ we conclude that $a=1$. We note  that $\acl(1)=\acl(\emptyset)$.

Let  $a^F  \in \acl^c(\bar h_{0}) \cap \acl^c(\bar h_{2})$. Therefore  $a^F\in H_{0}^c \cap H_{2}^c$.  Hence there exists $\a \in H_{0}, \b \in H_{2}$ such that $\a^F=\b^F=a^F$.  Suppose that $\a \neq 1$. Clearly  $\a$ and $\b$ are not conjugate in $P_{2}=H_0*H_1*H_2$. By applying \cite[Lemma 3.1]{hom-ould} it follows  that the pair  $(\alpha, \beta)$ is conjugate in  $P_2$ to a  one of the pairs  $(d_{0}^r, c_{1}^r), (c_{1}^r, d_{0}^r), (d_{1}^r, c_{2}^r), ( c_{2}^r, d_{1}^r)$ for some $r \in \mathbb Z$; this a contradiction.  Therefore  $a^F=1^F$. By Remark \ref{rem2} we have $\acl^c(\emptyset)=\acl^c(1)=\{1^F\}$ which concludes the proof. 
\end{proof}

\begin{lemma}  For $i\geq 0$ we have the following:
\[\acl(\overline h_0, \overline h_2, \dots, \overline h_{2(i-1)}, \overline h_{2i}) \cap \acl(\overline h_0, \overline h_2, \dots, \overline h_{2(i-1)}, \overline h_{2(i+1)})
=\acl\overline h_0, \overline h_2, \dots, \overline h_{2(i-1)})
\]
and 
\[
\acl^c(\overline h_0, \overline h_2, \dots, \overline h_{2(i-1)}, \overline h_{2i}) \cap \acl^c(\overline h_0, \overline h_2, \dots, \overline h_{2(i-1)}, \overline h_{2(i+1)})=\acl^c(\overline h_0, \overline h_2, \dots, \overline h_{2(i-1)}).
\]

\end{lemma}

\begin{proof} 

We may work in $F=G_{2(i+1)}$. The first result follows from Lemma~\ref{l-aclH} and normal forms: if $L=A*B*C$ then $A*B \cap A*C=A$.  
For the second part let $c^F \in \acl^c(\overline h_0, \overline h_2, \dots, \overline h_{2(i-1)}, \overline h_{2i}) \cap \acl^c(\overline h_0, \overline h_2, \dots, \overline h_{2(i-1)}, \overline h_{2(i+1)}), c\neq 1$. 

By Lemma~\ref{l-aclH} $(ii)$  there exist $\a \in \acl(\overline h_0, \overline h_2, \dots, \overline h_{2(i-1)}, \overline h_{2i})$ and $\b \in \acl(\overline h_0, \overline h_2, \dots, \overline h_{2(i-1)}, \overline h_{2(i+1)})$ such that $\a^F=\b^F=c^F$.  We have 
$$
 G_{2(i+1)}=\<G_{2(i-1)}*H_{2i-1}*H_{2i}*H_{2i+1}*H_{2(i+1)}, t_j, 2(i-1)\leq j\leq 2i+1\mid d_j^{t_j}=c_{j+1}\>. 
$$

First suppose that $\a$ and $\b$ are conjugate in $L=G_{2(i-1)}*H_{2i-1}*H_{2i}*H_{2i+1}*H_{2(i+1)}$. Since $\a \in G_{2(i-1)}*H_{2i}$ and $\b \in G_{2(i-1)}*H_{2(i+1)}$, it follows from properties of normal forms that $\a$ is conjugate to an element of $G_{2(i-1)}$. But since $\a \in \acl(\overline h_0, \overline h_2, \dots, \overline h_{2(i-1)})*H_{2i} $ it follows that $\a$ is conjugate to an element of  $\acl(\overline h_0, \overline h_2, \dots, \overline h_{2(i-1)})$; which is the required result. 

Now suppose that $\a$ and $\b$ are conjugate in $ G_{2(i+1)}$ but not in $L=G_{2(i-1)}*H_{2i-1}*H_{2i}*H_{2i+1}*H_{2(i+1)}$.  Then $\a$ (and similarly $\b$) is conjugate in $L$ to a power of one the elements $d_{2i-1}$, $c_{2i}$, $d_{2(i-1)}$, $c_{2i-1}$, $d_{2i}$, $c_{2i+1}$,  $d_{2i+1}$, $c_{2(i+1)}$. 

Since $\a \in G_{2(i-1)}*H_{2i}$, we conclude that $\a$ is conjugate to a power of $c_{2i}$ or $d_{2i}$ or $d_{2(i-1)}$. Similarly, since $\b \in  G_{2(i-1)}*H_{2(i+1)}$, we conclude that $\b$ is conjugate to a power of $c_{2(i+1)}$ or   $d_{2(i-1)}$. 

If $\a$ is conjugate to a power of $c_{2i}$ then $\b$ is conjugate (in $L$ ) to a power of $d_{2i-1}$; which is a contradiction. Similarly, if $\a$ is conjugate  to a power of $d_{2i}$ then $\b$ is conjugate (in $L$ ) to a power of $c_{2i+1}$; which is also a contradiction. Hence $\a$ is conjugate to $d_{2(i-1)}$ and thus $c^F \in \acl^c(\overline h_0, \overline h_2, \dots, \overline h_{2(i-1)})$ as required. 
\end{proof}

By Corollary~\ref{cor-acleq} we thus have proved the following:
\begin{corollary} \label{cor-prop-acl}
We have 
 \[\acleq(\overline h_0) \cap \acleq(\overline h_2)=\acleq(\emptyset)\]

and for $i\geq 1$ 
\[
\acleq(\overline h_0, \overline h_2, \dots, \overline h_{2(i-1)}, \overline h_{2i}) \cap \acleq(\overline h_0, \overline h_2, \dots, \overline h_{2(i-1)}, \overline h_{2(i+1)})
=\acleq(\overline h_0, \overline h_2, \dots, \overline h_{2(i-1)}). \qed
\]
\end{corollary}

\bigskip

To finish the proof of the fact that our sequence is a witness for the $n$-ample property  we prove  the two next lemmas which yield the required properties of independence.

\begin{lemma} For $i=1,\ldots, n-1$, there exists a free decomposition $G_n=K*H_{2i}*L$ such that $\overline h_0, \dots, \overline h_{2(i-1)} \in K*H_{2i}$ and $\overline h_{2(i+1)}\in H_{2i}*L$. 

\end{lemma}

\begin{proof} Let $1\leq i\leq n-1$ and put
$$
L_i=\<H_{2i+1}*\cdots *H_n,  t_j, 2i+1 \leq j \leq n-1\mid d_j^{t_j}=c_{j+1}\>.
$$
Then 
$$
G_n=\<G_{2i-1}*H_{2i}*L_i, t_{2i-1}, t_{2i} \mid d_{2i-1}^{t_{2i-1}}=c_{2i}, d_{2i}^{t_{2i}}=c_{2i+1}\>
$$

Since $H_{2i-1}$ is a free factor of $G_{2i-1}$, $d_{i-1}$ is primitive in $G_{2i-1}$ and thus we can write $G_{2i-1}=K_0*\<d_{2i-1}\>$ for some free group $K_0$. Similarly, $c_{2i+1}$ is primitive in $L_i$ and thus we can write $L_i=\<c_{2i+1}\>*L_0$ for some free group $L_0$. 

Therefore 
$$
G_n=K_0*\<t_{2i-1}\mid \>*H_{2i}*\<t_{2i}\mid \>*L_0,
$$
and by setting $K=K_0*\<t_{2i-1}\mid \>$ and $L=\<t_{2i}\mid \>*L_0$ we get $G_n=K*H_{2i}*L$ with  $\overline h_0, \dots, \overline h_{2(i-1)} \in K*H_{2i}$ and $\overline h_{2(i+1)}\in H_{2i}*L$ as required.
\end{proof}

\begin{lemma} There is no free decomposition $G_n=K*L$ such that $\overline h_0 \in K$ and $\overline h_n \in L$. 
\end{lemma}

\begin{proof} Suppose towards a contradiction that such a free decomposition exists. Since $H_i$ is a free factor, it follows that $H_0$ is a free factor of $K$ and $H_n$ is a free factor of $L$. Since $c_0$ is primitive in $H_0$ it is primitive in $K$ and since $d_n$ is primitive in $H_n$ it is primitive in $K$. We conclude that $\{c_0, d_n\}$ is part of a basis of $G_n$. 
Therefore in the abelianisation $G_n^{ab}=G_n/[G_n, G_n]$, we get that $\{c_0, d_n\}$ is part of a basis. 

However in $G_n^{ab}$ we have $c_id_i=1$ and $d_i=c_{i+1}$ and thus $c_0=d_n^{\pm 1}$ depending on whether $n$ is odd or even, a contradiction. \end{proof}

By Proposition \ref{p-forkingindep} we have thus proved the following:
\begin{corollary} \label{cor-propo-inde}
We have 
\[{\overline h_{2n}}  \NotInd{}{} {\overline h_0}\]
and for $i\geq 1$ 
\[\overline h_0\ldots \overline h_{2(i-1)} \Ind_{\overline h_{2i}}\overline h_{2(i+1)}. \qed\]
\end{corollary}

\bigskip
Putting Corollary \ref{cor-prop-acl} and Corollary \ref{cor-propo-inde} together  we therefore have  proved the following:

\begin{theorem}\label{thm-free}  The $u_i=\overline h_{2i}\in G_{2n},i=0,\ldots n$  witness the fact that $T_{fg}$ is $n$-ample,
i.e. we have the following:
\end{theorem}

\smallskip
 $(i)$  ${u_n}  \NotInd{}{} {u_0}$;

 $(ii)$    $u_0\ldots u_{i-1} \Ind_{u_i}u_{i+1}$ for $1 \leq i <n$;

$(iii)$ $\acleq(u_0) \cap \acleq(u_1)=\acleq(\emptyset)$. 

$(iv)$ $\acleq(u_0, u_1, \dots, u_{i-1}, u_{i}) \cap \acleq(u_0, u_1, \dots, u_{i-1}, u_{i+1})=\acleq(u_0, u_1, \dots, u_{i-1})$. \qed

\begin{remark}In fact, since
\[
G_{n+2}=\<G_n*H_{n+1}*H_{n+2}, t_j,j=n, n+1\mid d_j^{t_j}=c_{j+1}\>
\]
any free group $F$ of infinite rank
contains a sequence $(u_n\colon n<\omega)$ of tuples  such that

\smallskip
 $(i)$  ${u_i}  \NotInd{}{} {u_j}$  for $i\neq j$;

 $(ii)$  $\indep{u_0\ldots u_{i-1}}{u_i}{u_{i+1}}$

$(iii)$ $\acleq(u_0) \cap \acleq(u_1)=\acleq(\emptyset)$. 

$(iv)$ $\acleq(u_0, u_1, \dots, u_{i-1}, u_{i}) \cap \acleq(u_0, u_1, \dots, u_{i-1}, u_{i+1})=\acleq(u_0, u_1, \dots, u_{i-1})$. 

\smallskip
In particular,  $F_{\omega}$  contains  an explicit  sequence $(u_n\colon n<\omega)$ such that for every $n$ the finite sequence $u_0, u_1, \cdots, u_n$ is a witness of  the $n$-ampleness. 

\end{remark}

\section{Proof of the main theorem}

We now move back to working in a \tfh group. Let $H=\G*G_{2n}$ where
$\G$ is \tfh and $G_{2n}$ is as before.
In order to finish the proof of the main theorem we just need the
following observation:

\begin{lemma}\label{lemma-aclFH}With $\overline h_0,\overline h_2,\ldots , \overline h_{2n}$ defined as in Section 4, we have for $0\leq i\leq k$
\[\acl_{G_{2n}}(\overline h_0,\overline h_2,\ldots,\overline h_{2i},\overline h_{2k})=
\acl_H(\overline h_0,\overline h_2,\ldots,\overline h_{2i},\overline h_{2k}).\]
\end{lemma}
\begin{proof}
By Lemma~\ref{lemma-acl} we have for $0\leq i\leq k$
\[\acl_{G_{2n}}(\overline h_0,\overline h_2,\overline h_{2i}\overline h_{2k})=H_0*H_2*\ldots *H_{2i}*H_{2k}.\]

By Theorem~\ref{theorem-acl} we know that $H_0*H_2*\ldots *H_{2i}*H_{2k}$ is
the vertex group containing $\{\overline h_0,\overline h_2,\ldots,\overline h_{2i},\overline h_{2k}\}$ in the generalized malnormal cyclic $JSJ$-decomposition $\Lambda$ of $G_{2n}$ relative to $\{\overline h_0,\overline h_2,\ldots ,\overline h_{2i},\overline h_{2k}\}$.
Using $\Lambda$ we obtain  a malnormal cyclic splitting of $H=\G*G_{2n}$.
By Proposition~\ref{prop-spli-acl} we now see that 
\[\acl_H(\overline h_0,\overline h_2,\ldots,\overline h_{2i},\overline h_{2k})=H_0*H_2*\ldots *H_{2i}*H_{2k}.\]
\end{proof}

Corollary~\ref{cor-reduction} combined with Lemma \ref{l-aclH} and Theorem~\ref{thm-free} now yield our main theorem:

\begin{theorem} Let $\G$ be a \tfh group and $T=Th(\G)$. Consider
the model $H=EC(\G)*G_{2n}$ of $T$. Then $u_i=\overline h_{2i}\in G_{2n},i=0,\ldots n$  witness the fact that $T$ is $n$-ample,
i.e. we have the following:
\end{theorem}

\smallskip
 $(i)$  ${u_n}  \NotInd{}{} {u_0}$;

 $(ii)$    $u_0\ldots u_{i-1} \Ind_{u_i}u_{i+1}$ for $1 \leq i <n$;

$(iii)$ $\acleq(u_0) \cap \acleq(u_1)=\acleq(\emptyset)$. 

$(iv)$ $\acleq(u_0, u_1, \dots, u_{i-1}, u_{i}) \cap \acleq(u_0, u_1, \dots, u_{i-1}, u_{i+1})=\acleq(u_0, u_1, \dots, u_{i-1})$.

\bigskip
\noindent
{\bf Acknowledgement:} The authors would like to thank Zlil Sela, Anand Pillay and Gilbert Levitt for helpful discussions and Rizos Sklinos for having indicated to  us an omission in the first version of the paper.


\begin{thebibliography}{999}

\bibitem{gui-limit} C. Champetier and  V. Guirardel, Limit groups as 
limits of free groups: compactifying the set of free groups, Israel 
Journal of Mathematics, 146 (2005).


\bibitem{david-ample} D. Evans, Ample Dividing.  The Journal of Symbolic Logic
Vol. 68, No. 4 (2003), pp. 1385-1402. 

\bibitem{jsj-gl} V. Guirardel and  G. Levitt, JSJ decompositions: definitions, existence, uniqueness. 
I: The JSJ deformation space. 

\bibitem{khar-mia} O. Kharlampovich, A. Miasnikov, Elementary theory of free nonabelian groups,  J.Algebra, 302, Issue 2, 451-552, 2006.

\bibitem{Lynd-schup} R. C. Lyndon and P. E. Schupp, Combinatorial group theory, Springer, 1977. 


\bibitem{hom-ould}  A. Ould Houcine, Homogeneity and prime models in torsion-free hyperbolic groups. Confluentes Mathematici, 3 (1) (2011) 121-155.

\bibitem{OHV} A. Ould Houcine and  D. Vallino, Algebraic and definable
closure in free groups, preprint, 2011. 


\bibitem{paulin} F. Paulin : Topologie de Gromov \'equivariante, structures hyperboliques et arbres r\'eels. 
Invent. Math., 94(1):53�80, 1988. 


\bibitem{pillay-free} A. Pillay, Forking in the free group,
Journal of the Institute of Mathematics of Jussieu (2008), 7 : pp 375-389.

\bibitem{pillay-amp} A. Pillay, A note on CM-triviality and the geometry of forking. J. 
Symbolic Logic 65 (2000), no. 1, 474-480.

\bibitem{pillay-weight} A. Pillay, On genericity and weight in the free   
group,
Proceedings of the AMS (2009), 11 : pp 3911-3917.


\bibitem{perin-homo}
C.~Perin and R.~Sklinos.
\newblock Homogeneity in the free group.
\newblock {\em Preprint}, 2010.

\bibitem{perin-homo-personal}
C.~Perin and R.~Sklinos.
\newblock unpublished notes.





\bibitem{sela-JSJ}
E.~Rips and Z.~Sela.
\newblock Cyclic splittings of finitely presented groups and the canonical
  {JSJ} decomposition.
\newblock {\em Ann. of Math. (2)}, 146(1):53--109, 1997.




\bibitem{sela-el} Z. Sela, Diophantine geometry over groups VI: The elementary theory of a free group,  GAFA 16(2006), 707-730.


\bibitem{sela09} Z. Sela, Diophantine Geometry over Groups VII. The elementary theory of a hyperbolic group. Proc. Lond. Math. Soc. (3), 99(1):217--273,2009.

\bibitem{sel-ima} Z. Sela, Diophantine Geometry over Groups IX: Envelopes and Imaginaries, ArXiv e-prints, 2009. 

\bibitem{sel-free-product} Z. Sela, Diophantine geometry over groups X: The Elementary Theory of Free Products of Groups. 
\bibitem{TZ} K. Tent, M. Ziegler,  A course in model theory, Lecture Notes in Logic, Cambridge University Press, 2012.



\end{thebibliography}
\end{document}